\documentclass[11pt,a4paper]{article}
\usepackage{style}
\usepackage{booktabs}
\usepackage{array}

\usepackage{textcmds} 
\usepackage{amsmath}
\usepackage{amstext}
\usepackage{amsthm}
\usepackage{amsfonts}
\usepackage{enumitem}
\usepackage{amssymb}
\usepackage{xspace}
\usepackage{microtype}
\usepackage[all]{xy}
\usepackage[modulo]{lineno}

\usepackage{dsfont}
\usepackage{bm}
\usepackage{subfigure}
\usepackage{array}
\usepackage[all]{xy}
\usepackage{euscript}
\usepackage[T1]{fontenc}
\usepackage{mathbbol}
\usepackage{nccmath}
\usepackage{marginnote}
\usepackage{stmaryrd}
\usepackage{tikz}
\usepackage{tikz-cd} 
\usepackage{graphics,graphicx}  

\usepackage[draft]{say}

\newcommand{\nt}{\op{int}}
\newcommand{\wt}{\widetilde}
\newcommand{\cech}{c^{\op{ECH}}}
\newcommand{\pr}{\op{pr}}
\newcommand{\id}{\op{Id}}
\newcommand{\pt}{\op{pt}}

\title{On Symplectic Packing Problems in Higher Dimensions}

\date{\today}
\author{Kyler Siegel, Yuan Yao}

\begin{document}
\maketitle

\begin{abstract}
Let $B^{2n}(R)$ denote the closed $2n$-dimensional symplectic ball of area $R$,
and let $\Sigma_g(L)$ be a closed symplectic surface of genus $g$ and area $L$.
We prove that there is a symplectic embedding $\bigsqcup\limits_{i=1}^k B^4(R_i) \times \Sigma_g (L) \hooksymp \nt(B^4(R))\times \Sigma_g (L)$ if and only if there exists a symplectic embedding $\bigsqcup\limits_{i=1}^k B^4(R_i) \hooksymp \nt(B^4(R))$. 

This lies in contrast with the standard higher dimensional ball packing problem $\bigsqcup\limits_{i=1}^k B^{2 n}(R_i) \hooksymp \nt(B^{2n}(R))$ for $n >2$, which we conjecture (based on index behavior for pseudoholomorphic curves) is controlled entirely by Gromov's two ball theorem and volume considerations.
We also deduce analogous results for stabilized embeddings of concave toric domains into convex domains, and we establish a stabilized version of Gromov's two ball theorem which holds in any dimension.
Our main tools are: (i) the symplectic blowup construction along symplectic submanifolds, (ii) an h-principle for symplectic surfaces in high dimensional symplectic manifolds, and (iii) a stabilization result for pseudoholomorphic holomorphic curves of genus zero.

\end{abstract}

\tableofcontents
 
\section{Introduction}
\subsection{Background and main results}

The symplectic ball packing problem, which asks when a disjoint union of $2n$-dimensional symplectic balls of prescribed sizes can be symplectically embedded into a fixed $2n$-dimensional symplectic ball, is similar in spirit but quite distinct in substance from its more famous Riemannian cousin, namely the sphere packing problem (see e.g. \cite{conway2013sphere}).
Let $B^{2n}(R) := \left \{\pi \sum\limits_{i=1}^n|z_i|^2 \leq R \right\} \subset \C^n$ denote the closed  $2n$-dimensional ball of area $R$ (i.e. radius $\sqrt{R/\pi})$, and let $\nt(B^{2n}(R))$ denote its interior.
We endow these with the restriction of the standard symplectic two-form $\omega_\std = \sum\limits_{i=1}^n dx_i \wedge dy_i$.
Classically, note that any symplectic embedding
$\bigsqcup\limits_{i=1}^k B^{2n}(R_i) \hooksymp \nt(B^{2n}(R))$
must satisfy 
$\sum\limits_{i=1}^k R_i^n \leq R^n$,
 since symplectic embeddings preserve the standard volume form $\tfrac{1}{n!}\omega_\std^{\wedge n}$.
The first nontrivial result on symplectic ball packings appeared in Gromov's foundational 1985 paper on pseudoholomorphic curves.
\begin{theorem}[Gromov's two ball theorem \cite{gromov1985pseudo}]\label{thm:two_ball}
Given $n \in \Z_{\geq 1}$ and $R_1,R_2,R \in \R_{>0}$, any symplectic embedding $B^{2n}(R_1) \bigsqcup B^{2n}(R_2) \hooksymp \nt(B^{2n}(R))$ must satisfy $R_1 + R_2 < R$.
\end{theorem}

Subsequently, in dimension four many more refined symplectic packing obstructions for multiple balls were found.
Notably, in the special case of balls of equal sizes, the supremal fraction $v(B^4,k)$ of the volume of $B^4 := B^4(1)$ that can be filled by a symplectic ball packing $\bigsqcup\limits_{i=1}^k B^4(R) \hooksymp \nt(B^4)$ for some $R \in \R_{>0}$ was computed in the combined works \cite{gromov1985pseudo,mcduff1994symplectic,biran1997symplectic} to be:

\begin{equation}\label{eq:equal_ball_nums}
 \begin{array}{|c|c|c|c|c|c|c|c|c|c|}
 \hline
k & 1 & 2 & 3 & 4 & 5 & 6 & 7 &8 & \geq 9\\ \hline 
\rule{0pt}{2.5ex} v(B^4,k) & 1 & \tfrac{1}{2} & \tfrac{3}{4} & 1 & \tfrac{20}{25} & \tfrac{24}{25} & \tfrac{63}{64} & \tfrac{288}{289}& 1\\[.5ex]\hline
 \end{array}\,.
\end{equation}
More generally, building on various works \cite{mcduff1994symplectic,mcduff_from_def_to_iso,biran1997symplectic,biran2001symplectic,li2001uniqueness,McDuff_2009,McDuff_Hofer_conjecture,hutchings_quant_ech}, the symplectic ball packing problem  in dimension $2n=4$ has been entirely computed, or at least reduced to combinatorics.
\begin{theorem}[{see \cite[Rmk. 1.10]{hutchings_quant_ech}}]\label{thm:combinatorial_packing_obstructions} 
There exists a symplectic embedding $\bigsqcup\limits_{i=1}^k B^4(R_i) \hooksymp \nt( B^4(R))$ if and only if for each\footnote{In fact, using \cite{mcduff1994symplectic,li_li_2002} it suffices to consider only certain distinguished tuples related to the Cremona group -- see \cite[\S1.2]{Mcduff_Schlenk_infinite_staircase}.} tuple of nonnegative integers $(d;m_1,\dots,m_k)$ (not all zero) satisfying 
\[
\sum_{i=1}^k (m_i^2+m_i) \leq d^2 +3d
\] 
we have
\[
\sum_{i=1}^k m_i R_i < dR.
\]
\end{theorem}
\NI 
The obstructions involved in Theorem~\ref{thm:combinatorial_packing_obstructions} can be deduced either from exceptional curves in blowups of $\CP^2$, or using the embedded contact homology (ECH) capacities.

Symplectic embedding problems are well-studied in dimension four, partly due to the existence of powerful gauge theoretic invariants like Seiberg--Witten theory, and partly thanks to useful tools for pseudoholomorphic curves in low dimensions such as positivity of intersections and automatic transversality.
For comparison, in higher dimensions we have the following conjecture concerning symplectic ball packings.

\begin{conjlet}\label{conj:high_d_ball_packing}
For $n \geq 3$ and $R_1,\dots,R_k,R \in \R_{>0}$ for some $k \in \Z_{\geq 1}$, there exists a symplectic embedding
\begin{align*}
\bigsqcup_{i=1}^k  B^{2n}(R_i) \hooksymp \nt(B^{2n}(R))
\end{align*}
if and only if the following hold:
\begin{itemize}
  \item (volume obstruction) $R_1^n + \cdots + R_k^n < R^n$ 
  \item (two ball obstruction) $R_i + R_j < R$ for all $1 \leq i < j \leq k$.
\end{itemize}
\end{conjlet}

\NI Indeed, for essentially formal reasons (based on a Fredholm index computation in \S\ref{section:higher_dim_packing}) there cannot be any stronger obstructions coming from pseudoholomorphic curves, at least under the standard paradigm for obstructing symplectic embeddings via closed or punctured curves.
In the spirit of Y. Eliashberg's ``holomorphic curves or nothing'' philosophy, which posits that all nontrivial obstructions in symplectic geometry come from some kind of pseudoholomorphic invariants or classical invariants, we then expect higher dimensional ball packings to exhibit complete flexibility beyond the volume and two ball obstructions (see \S\ref{section:higher_dim_packing} for more on what is currently known).

Nevertheless, we show that a certain stabilization procedure exhibits considerable rigidity for symplectic packing problems in higher dimensions.  Fixing $g \in \Z_{\geq 1}$ and $L \in \R_{>0}$, let $\Sigma_g(L)$ denote a closed symplectic surface of genus $g$ and area $L$. 
Given a symplectic ball $B^{2n}(R)$, we equip $B^{2n}(R) \times \Sigma_g(L)$ with the product symplectic form.

Our main result is:
\begin{thmlet}\label{thm:main_thm}
There exists a symplectic embedding $\bigsqcup\limits_{i=1}^k B^4(R_i)\times \Sigma_g (L) \hooksymp \nt(B^4(R)) \times \Sigma_g (L)$ if and only if there exists a symplectic embedding $\bigsqcup\limits_{i=1}^k B^4(R_i) \hooksymp \nt(B^4(R))$. 
\end{thmlet}
\NI In particular, the same numbers appearing in \eqref{eq:equal_ball_nums} give the supremal volume of $B^4 \times \Sigma_g(L)$ which can be filled by a packing $\bigsqcup\limits_{i=1}^k B^4(R) \times \Sigma_g(L) \hooksymp \nt(B^4) \times \Sigma_g(L)$ for some $R \in \R_{>0}$ (for any fixed $L \in \R_{>0}$).
We summarize the proof strategy for Theorem~\ref{thm:main_thm} in \S\ref{section:Sketch_of_proof}, with the details appearing in \S\ref{sec:pf_main_thm}.

Using similar ideas as Theorem \ref{thm:main_thm}, we also deduce a ``stabilized'' version of Gromov's two ball theorem, which holds in any dimension:
\begin{thmlet}\label{thm:stabilized_Gromov_two_ball}
Let $N$ be a closed
 symplectic manifold which is symplectically aspherical.\footnote{Recall that a symplectic manifold $(M,\omega)$ is {\em symplectically aspherical} if $\omega$ vanishes on $\pi_2(M)$. We expect this assumption could be removed at the cost of using virtual fundamental class techniques for Gromov--Witten theory.} Then there exists a symplectic embedding $\left(B^{2n}(R_1) \times N\right) \sqcup \left(B^{2n} (R_2) \times N\right) \hooksymp \nt(B^{2n}(R)) \times N$ if and only if there exists a symplectic embedding $B^{2n}(R_1) \sqcup B^{2n} (R_2) \hooksymp \nt(B^{2n}(R))$.
\end{thmlet}

We give a general discussion of what our methods produce for higher dimensional stabilized problems in Remark  \ref{rmk:stab_in_higher_dimensions}.

\subsubsection*{Other stabilized embedding problems}

Since ball packings are known to control various symplectic embeddings problems, we can use Theorem~\ref{thm:main_thm} to deduce analogous embedding results for stabilizations of other domains.
McDuff \cite{McDuff_2009} showed that each symplectic embedding problem between four-dimensional rational ellipsoids is equivalent to some symplectic ball packing problem, where the sizes of the balls are read off from the area factors of the ellipsoids. 
Cristofaro-Gardiner subsequently generalized this to a large class of four-dimensional toric domains as follows: 

\begin{theorem}[{\cite[Thm. 2.1]{DanCG_Concave_into_Convex}}]\label{thm:Dan_concave_into_convex}
Given a four-dimensional rational concave toric domain $X_{\Omega_1}^4$ and a four-dimensional rational convex toric domain $X_{\Omega_2}^4$, there exists a symplectic embedding $X_{\Omega_1} \hooksymp \nt(X_{\Omega_2})$ if and only if there exists a ball packing
\begin{align}
\left(\bigsqcup\limits_{i=1}^k B^4(a_i)\right) \sqcup \left(\bigsqcup\limits_{i=1}^\ell B^4(b_i)\right) \hooksymp \nt(B^4(b)),
\end{align}
where $(a_1,\dots,a_k)$ is the weight sequence of $\Omega_1$ and $(b;b_1,\dots,b_\ell)$ is the negative weight sequence of $\Omega_2$.
\end{theorem}

We briefly recall the definitions of (rational) concave and convex toric domains and their weight sequences in \S\ref{sec:stabilized_toric}.
The basic idea is that a four-dimensional concave toric domain canonically decomposes into a collection of symplectic balls, and similarly for the complement of a four-dimensional convex toric domain inside of a larger circumscribed ball.
By adapting this picture to the stabilizations $X_{\Omega_i} \times \Sigma_g(L)$ and applying Theorem~\ref{thm:main_thm}, we get the following stabilized analogue of Theorem~\ref{thm:Dan_concave_into_convex}.
\begin{thmlet} \label{thm:stabilized_toric}
Let $X_{\Omega_1}^4$ be a four-dimensional concave toric domain, and let $X_{\Omega_2}^4$ be a four-dimensional convex toric domain. Let $\Sigma_g(L)$ be a closed symplectic surface of area $L$ and genus $g$. Then there exists a symplectic embedding 
\[
\nt(X_{\Omega_1}) \times \Sigma_g(L) \hooksymp \nt(X_{\Omega_2}) \times \Sigma_g(L)
\]
if and only if there exists a symplectic embedding
\[
\nt(X_{\Omega_1}) \hooksymp \nt(X_{\Omega_2}).
\]
\end{thmlet}

Combining this with the main result of \cite{DanCG_Concave_into_Convex}, which states that the ECH capacities give sharp obstructions for embeddings of concave into convex toric domains,

we have:
\begin{corlet}\label{cor:ECH_caps_sharp_for_stab_prob}
For $X_{\Omega_1}^4,X_{\Omega_2}^4$ as in Theorem~\ref{thm:stabilized_toric}, there exists a symplectic embedding 
\begin{align}\label{eq:intro_conc_into_conv_emb}
\nt(X_{\Omega_1}) \times \Sigma_g(L) \hooksymp \nt(X_{\Omega_2}) \times \Sigma_g(L)
\end{align}
 if and only if we have \[\cech_k(X_{\Omega_1}) \leq \cech_k(X_{\Omega_2})\] for all $k \in \Z_{\geq 1}$.
\end{corlet}
\NI The ECH capacities of $X_{\Omega_1}$ and $X_{\Omega_2}$ have purely combinatorial descriptions -- see \cite[\S A]{DanCG_Concave_into_Convex} and the references therein for more details.
It is noteworthy that the ECH capacities control the six dimensional embedding problem \eqref{eq:intro_conc_into_conv_emb}, even though they are only defined in dimension four.

Specializing to ellipsoids $E(a,b) := \{\pi|z_1|^2/a +\pi|z_2|^2/b\leq 1\} \subset \C^2$ (these are both concave and convex toric domains), embeddings of the form $E(a,b) \hooksymp \nt(B^4)$ were studied extensively in \cite{Mcduff_Schlenk_infinite_staircase}, where the function 
\[
f_{B^4}(x) = \inf \{\mu  \;|\; E(1,x) \hooksymp \nt(B^4(\mu))\}
\]
was explicitly computed and shown to contain an ``infinite staircase'' controlled by the Fibonacci numbers. By Theorem~\ref{thm:stabilized_toric}, this same function persists after stabilizing by $\Sigma_g(L)$.
\begin{corlet}\label{cor:MS_func_stabilizes}
For each $x \in \R_{\geq 1}$, there exists a symplectic embedding 
\[E(1,x) \times \Sigma_g(L) \hooksymp \nt(B^4(\mu)) \times \Sigma_g(L)\] if and only if we have $\mu > f_{B^4}(x)$.
\end{corlet}

Note that there has been considerable recent interest in pushing four dimensional symplectic  embedding results into higher dimensions by a different type of stabilization, namely taking products with $\R^2$, or more generally $\R^{2N}$.
By analogy with the definition of $f_{B^4}(x)$, for $N \in \Z_{\geq 1}$ put
\[
f_{B^4 \times \R^{2N}}(x) := \inf \{\mu  \;|\; E(1,x) \times \R^{2N} \hooksymp \nt(B^4(\mu)) \times \R^{2N}\}.
\]
It is known \cite{Hind_Kerman,CG_Hind_embedding_products,ghost_staircase,Mint,chscI} that, for any $N \in \Z_{\geq 1}$, $f_{B^4 \times \R^{2N}}(x)$ coincides with $f_{B^4}(x)$ for all $1 \leq x \leq \tau^4 := \tfrac{7+3\sqrt{5}}{2}$, while $f_{B^4 \times \R^{2N}}(x) = \tfrac{3x}{x+1}$ for infinitely many values of $x$ in $(\tau^4,\infty)$. 

\NI It is a subtle problem to determine which four-dimensional obstructions persist after stabilizing with $\R^{2N}$ and which do not. For instance, for $x$ sufficiently large we have $f_{B^4}(x) = \sqrt{x}$, i.e. volume is the only obstruction, while conjecturally (and definitely for certain values) we have $f_{B^4 \times \R^{2N}}(x) = \tfrac{3x}{x+1}$, coming from Hind's folding embedding \cite{Hind_optimal_embedding_ellipsoids}.
By contrast, Corollary~\ref{cor:MS_func_stabilizes} states that all four-dimensional obstructions persist after stabilizing with $\Sigma_g(L)$.

\begin{remark}
As another illustration of the difference between stabilizing with $\Sigma_g(L)$ versus $\R^{2N}$, let us note that in the latter case the analogue of Theorem~\ref{thm:main_thm} is entirely flexible.
Namely, for any $k \in \Z_{\geq 1}$ and $R_1,\dots, R_k,R \in \R_{>0}$ with $R_1,\dots,R_k \leq R$ there exists a symplectic embedding
    \[
    \bigsqcup_{i=1}^k B^4(R_i) \times \mathbb{R}^{2N} \hooksymp B^4(R) \times \mathbb{R}^{2N}.
    \] 
This is a consequence of the fact that we can symplectically embed a disjoint union of finitely many copies of $\R^{2N}$ into a single copy of $\R^{2N}$.
Note that the condition $R_1,\dots,R_k \leq R$ is necessary by Gromov's nonsqueezing theorem.
\end{remark}

\begin{remark}\label{rmk_open_vs_closed_domains}
In the embedding problems discussed above one could take the source to be either a compact domain or its interior, and similarly for the target, with the distinction between these choices being mostly academic (up to swapping inequalities with strict inequalities) since we are interested in {\em optimal} embeddings.
For a more refined discussion see also \cite[Rmk. 1.3]{McDuff_Hofer_conjecture} or \cite[Cor. 1.6]{DanCG_Concave_into_Convex}.
\end{remark}

\subsection{Open problems}
We now mention a few open problems which seem to go beyond the techniques of this paper.

\subsubsection*{More general stabilizations by symplectic surfaces}

The stabilized packing problem addressed by Theorem~\ref{thm:main_thm} has the following natural extension.
For $k \in \Z_{\geq 1}$, let $\Sigma_{g_1}(L_1),\dots,\Sigma_{g_k}(L_k)$ denote closed symplectic surfaces with genera $g_1,\dots,g_k \in \Z_{\geq 0}$ and areas $L_1,\dots,L_k \in \R_{>0}$ respectively.
\begin{question}
When does there exist a symplectic packing of the form
\begin{align*}
\bigsqcup\limits_{i=1}^k B^4(R_i) \times \Sigma_{g_i}(L_i) \hooksymp \nt(B^4(R)) \times \Sigma_{g}(L),
\end{align*}
for $R_1,\dots,R_k,R \in \R_{>0}$?
\end{question}

Note that by homology class considerations we must have $L_i = \kappa_i L$ for some $\kappa_i \in \Z_{\geq 1}$, for each $i = 1,\dots,k$.

\begin{example}
Let $S^2(L)$ denote the standard symplectic two-sphere normalized to have area $L \in \R_{>0}$, and 
consider the embedding problem 
\begin{align}\label{ex:kappa_L_into_L}
B^4(R) \times S^2(\kappa L) \hooksymp \nt(B^4(R')) \times S^2(L),
\end{align}
 for some fixed $\kappa \in \Z_{\geq 2}$.
By taking a generic smooth perturbation of the composition of the degree $\kappa$ branched covering $\CP^1 \ra \CP^1$, $z \mapsto z^\kappa$, with the inclusion $\CP^1 = \CP^1 \times \{0\}  \subset \CP^1 \times B^4$, we produce a symplectic submanifold of $\nt(B^4(R')) \times S^2(L)$ having area $\kappa L$.
Using the Weinstein neighborhood theorem, this extends to a symplectic embedding
$B^4(R) \times S^2(\kappa L) \hooksymp B^4(R') \times S^2(L)$ for some small $R > 0$.
It is interesting to ask how to construct such embeddings with maximal $R$.

\end{example}

\begin{remark}
    Note the analogue of \eqref{ex:kappa_L_into_L} in dimension four is vacuous, since symplectic embeddings of $S^2(\kappa L)$ into $\R^2 \times S^2(L)$ are prohibited by the adjunction formula.
\end{remark}

\subsubsection*{Stabilizing by higher dimensional manifolds}

In Theorem~\ref{thm:main_thm}, the fact that $\Sigma_g(L)$ is two-dimensional allows us to appeal to an h principle for symplectic surfaces in symplectic manifolds of dimension at least six (see \ref{thm:h_princ_symp_sub}).
As we explain in \S\ref{sec:h_principle}, a similar h principle holds for higher dimensional symplectic submanifolds of codimension at least four, but it is no longer guaranteed that the relevant formal obstructions vanish.
Nevertheless, we can ask:
\begin{question}
Let $N$ be a closed symplectic manifold of dimension $2n \geq 4$. Is it still true that $\bigsqcup\limits_{i=1}^k B^4(R_i) \times N \hooksymp B^4(R) \times N$ if and only if  $\bigsqcup\limits_{i=1}^k B^4(R_i) \hooksymp B^4(R)$?
\end{question}

\subsection{Sketch of proof of main theorem} \label{section:Sketch_of_proof}

We end this introduction by giving a brief outline of the proof of Theorem \ref{thm:main_thm}, while also outlining the main components of the paper. 
The ``if'' direction is immediate, since a four-dimensional embedding can be stabilized to a six-dimensional one. The proof of the ``only if'' direction is based on the symplectic ball packing obstructions coming from exceptional curves in blowups of $\CP^2$, following \cite{mcduff1994symplectic}. 
It relies on the following three observations.

\begin{itemize}
\item First we note that the existence of pairwise disjoint embeddings $\iota_i: B^4(R_i) \times \Sigma_g(L)\hooksymp \nt(B^4(R)) \times \Sigma_g(R)$ for $i = 1,\dots,k$ implies the existence of a symplectic blowup along the collection of symplectic surfaces $\iota_1(\{0\} \times \Sigma_g(L)),\dots,\iota_k(\{0\} \times \Sigma_g(L))$ in $\CP^2(R) \times \Sigma_g(L)$ with respective sizes $R_1,\dots,R_k$ (here $\CP^2(R)$ is equipped with the Fubini--Study form normalized so that a line has area $R$).
We denote this blowup by 
$\bl_{\{\iota_i\}}(\CP^2(R) \times \Sigma_g(L))$. 
The theory of symplectic blowup along submanifolds is reviewed in \S\ref{sec:symplectic_blow_up}.
\item Next, with the aid of an h principle, we can symplectically isotope the collection of symplectic surfaces 
$\iota_1(\{0\} \times \Sigma_g(L)),\dots,\iota_k(\{0\} \times \Sigma_g(L))$ in $\CP^2(R) \times \Sigma_g(L)$ to a collection of fibers $\{p_1\} \times \Sigma_g(L),\dots,\{p_k\} \times \Sigma_g(L)$ for some pairwise distinct $p_1,\dots,p_k \in \CP^2(R)$. 
We then find a symplectic deformation equivalence between $\bl_{\{\iota_i\}}(\CP^2(R) \times \Sigma_g(L))$ and the size $\eps$ blowup along the fibers $\{p_1\} \times \Sigma_g(L),\dots,\{p_k\} \times \Sigma_g(L)$ for some small $\eps > 0$,
which we denote by $\bl_{\{\hat{j}_i\}}(\CP^2(R) \times \Sigma_g(L))$ to be consistent with our later notation. 

\item  Finally, we observe in \S\ref{sec:h_principle} that the exceptional curves in blowups of $\CP^2(R)$ which give rise to packing obstructions in dimension four stabilize after crossing with $\Sigma_g(L)$. In particular, they are seen by Gromov-Witten invariants of $\bl_{\{\hat{j}_i\}}(\CP^2(R) \times \Sigma_g(L))$, and hence also of $\bl_{\{\iota_i\}}(\CP^2(R) \times \Sigma_g(L))$ by deformation invariance. Then the fact that the symplectic form must evaluate positively on these stabilized exceptional curves in $\bl_{\{\iota_i\}}(\CP^2(R) \times \Sigma_g(L))$ gives precisely the embedding obstructions of Theorem \ref{thm:main_thm}.

\end{itemize}

\subsection*{Acknowledgements}
Y.Y. would like to thank Dusa McDuff and Tian-Jun Li for helpful conversations. Y.Y. is partially supported by ERC Starting Grant No. 851701. K.S. is partially supported by NSF grant DMS-2105578.

\section{Blowing up along symplectic submanifolds} \label{sec:symplectic_blow_up}

In this section we discuss the symplectic blowup construction along a symplectic submanifold, highlighting the main invariance properties we will need for the proof of Theorem~\ref{thm:main_thm}.
While the symplectic blowup at a point is well-studied (see e.g. \cite{mcduff1994symplectic} or \cite[\S7.1]{mcduff2017introduction}), the analogue for symplectic submanifolds is somewhat less explored. Here we roughly follow the outline given in  \cite[\S7.1]{mcduff2017introduction}. 

After reviewing the blowup construction along a symplectic ball embedding $B^{2n}(R) \hooksymp M^{2n}$, we extend this fiberwise to construct the symplectic blowup along an embedding $B^{2k}(R) \times \Sigma^{2n-2k} \hooksymp M^{2n}$.
The properties we will need of this construction are summarized in Proposition~\ref{prop:main_properties_of_blowup},
whose proof is provided in \S\ref{sec:detailed_blowup}.

\subsection{The symplectic blowup at a point}\label{subsec:pt_blowup}

In order to set up notation let us first briefly review the symplectic blowup at a point $p$ in a symplectic manifold $M^{2n}$. Recall that this depends on a size parameter $R \in \R_{>0}$ and a symplectic embedding $\iota: B^{2n}(R) \hooksymp M$.
Roughly speaking, we obtain a new symplectic manifold $\wt{M}$ by removing the interior of $\iota(B^{2n}(R))$ and collapsing its boundary along the fibers of the Hopf fibration to obtain an exceptional divisor $Z \subset \wt{M}$ diffeomorphic to $\CP^{n-1}$.
Here the blowup parameter $R$ translates into $Z$ having symplectic form $R \,\omega_{\op{FS}}$, where $\omega_{\op{FS}}$ denotes the Fubini--Study form on $\CP^{n-1}$ normalized so that a line has unit area.

More precisely, let $\wt{\C}^n$ denote the complex analytic blowup of $\C^n$ at the origin (i.e. the total space of $\mathcal{O}(-1) \ra \CP^{n-1}$), with its natural projection maps $\pi: \wt{\C}^n \ra \C^n$ and $\pr: \wt{\C}^n \ra \CP^{n-1}$.
The exceptional divisor is $Z := \pi^{-1}(0)$, and we sometimes denote this by $\CP^{n-1} \subset \wt{\C}^n$ when no confusion should arise.
Given $R \in \R_{>0}$, we equip $\wt{\C}^n$ with the symplectic form $\wt{\omega}_R := \pi^* \omega_\std + R\, \pr^*\omega_{\op{FS}}$, where $\omega_\std = \sum_{i=1}^n dx_i\wedge dy_i$ denotes the standard symplectic form on $\C^n$.
Let $\wt{B}^{2n}(R)$ denote the preimage of $B^{2n}(R)$ under the projection $\pi: \wt{\C}^n \ra \C^n$.
Following \cite[Lem. 7.1.11]{mcduff2017introduction}, there is a diffeomorphism $F_R: \wt{\C}^n \setminus Z \ra \C^n \setminus B^{2n}(R)$ which restricts to a symplectomorphism 
\begin{align}
(\wt{B}^{2n}(\eps) \setminus Z,\wt{\omega}_R) \cong (B^{2n}(R+\eps) \setminus B^{2n}(R),\omega_\std)
\end{align}
for each $\eps > 0$.

Now suppose that we have a symplectic embedding $\iota: B^{2n}(R) \hooksymp M$.
In this situation we define the symplectic blowup by first extending $\iota$ to a symplectic embedding $B^{2n}(R + \eps) \hooksymp M$ for some small $\eps > 0$ (c.f. \cite[Thm. 3.3.1]{mcduff2017introduction}), and then putting 

\begin{align}
\bl_{\iota} M := \left(M \setminus \iota(B^{2n}(R)) \sqcup \wt{B}^{2n}(\eps)\right) / \sim,
\end{align}
 where we identify each point in $\wt{B}^{2n}(\eps) \setminus Z$ with its image under $\iota \circ  F_R$.
One can check that, up to symplectomorphism, $\bl_{\iota} M$ is independent of the extension of $\iota$ from $B^{2n}(R)$ to $B^{2n}(R + \eps')$, and is also unchanged if we deform $\iota$ through symplectic embeddings $B^{2n}(R) \hooksymp M$ (see \cite[Thm. 7.1.23]{mcduff2017introduction}).

\subsection{Blowing up a symplectic submanifold with trivial normal bundle}
\label{sec:blow_up_trivial_normal_bundle}

We now generalize the above discussion to symplectic blowups along symplectic submanifolds as follows. 
Let $(M^{2n},\omega)$ be a symplectic manifold and $\Sigma^{2n-2k} \subset M$ a symplectic submanifold of codimension $2k$. We henceforth assume $k\geq 2$. 
Recall that the germ of the symplectic form near $\Sigma$ is determined up to symplectomorphism by the symplectic normal bundle over $\Sigma$ (see \cite[Thm. 3.4.10]{mcduff2017introduction}).
In particular, when the symplectic normal bundle is trivial there is a small neighborhood of $\Sigma$ which is symplectomorphic to $ B^{2k}(\delta) \times \Sigma $ for some $\delta \in \R_{>0}$ (here we equip $\Sigma$ with the induced symplectic form $\omega|_{\Sigma}$ and $B^{2k}(\delta)$ with $\omega_\std$).

\begin{definition}[see also \cite{McDuff_simply_connected}]\label{def:blow_up_R_simple} 
Let $M^{2n}$ and $\Sigma^{2n-2k}$ be closed symplectic manifolds, and let $\iota:  B^{2k}(R) \times \Sigma \hooksymp M$
be a symplectic embedding for some $R \in \R_{>0}$.
We define the symplectic blowup along $\iota$ of size $R$ by first extending $\iota$ to a symplectic embedding
$B^{2k}(R + \eps) \times \Sigma  \hooksymp M$ for some small $\eps > 0$, and then putting
\begin{align}\label{eq:blowup_def}
\bl_{\iota}(M) := \left(M \setminus \iota( B^{2k}(R) \times \Sigma) \sqcup  ( \wt{B}^{2k}(\eps)\times \Sigma) \right) / \sim,
\end{align}
where we identify each point in $ \wt{B}^{2k}(\eps)\times \Sigma$ with its image under $\iota \circ (F_R \times \id )$.
\end{definition}

Note that in Definition~\ref{def:blow_up_R_simple} the symplectic submanifold $\iota(\{0\} \times \Sigma ) \subset M$ necessarily has trivial symplectic normal bundle. 
We will refer to $\bl_{\iota}(M)$ as a {\em blowup along $\iota(\{0\}\times  \Sigma )$ of size $R$}.

The construction in Definition~\ref{def:blow_up_R_simple} also naturally generalizes to the case of pairwise disjoint embeddings $\iota_i:  B^{2k}(R_i) \times \Sigma \hooksymp M$ for $i =1,\dots,k$ and some $R_1,\dots,R_k \in \R_{>0}$ (say by blowing up along the embeddings one at a time), and in this case we denote the corresponding blowup by $\bl_{\{\iota_i\}}(M)$.

 Observe that there is a natural symplectomophism 
\begin{align*}
 \beta: \bl_\iota(M)\setminus ( \CP^{k-1}\times \Sigma) \rightarrow M\setminus \iota( B^{2k}(R)\times \Sigma),
\end{align*}
and this induces a homology-level isomorphism 
 \begin{align}\label{eq:bu_hom_iso_on_comp}
 \beta_*^{-1}: H_2(M \setminus \iota( B^{2k}(R))\times \Sigma) \ra H_2(\bl_\iota(M)\setminus ( \CP^{k-1}\times \Sigma)).
 \end{align}
We also have an inclusion-induced maps $H_2(M \setminus \iota( B^{2k}(R)\times \Sigma)) \ra H_2(M)$ and $H_2(\bl_\iota(M)\setminus ( \CP^{k-1}\times \Sigma)) \ra H_2(\bl_\iota(M))$. The first of these is an isomorphism provided that $k \geq 2$, in which case together with \eqref{eq:bu_hom_iso_on_comp} we get a map $H_2(M) \ra H_2(\bl_\iota(M))$ which we denote again by $\beta_*^{-1}$. This map is injective since it splits the projection map $H_2(\bl_\iota(M)) \ra H_2(M)$.

The following proposition lists the main properties we will need for the blowup construction.

\begin{proposition} \label{prop:main_properties_of_blowup}
Let $\bl_{\iota}(M)$ denote the symplectic blowup along an embedding $\iota:  B^{2k}(R) \times \Sigma^{2n-2k} \hooksymp M$ as in Definition~\ref{def:blow_up_R_simple}. 

\begin{enumerate}[label=(\Alph*)]
\item \label{item:homology_descr}
We have an isomorphism
\begin{equation} \label{eq:homology_of_blow_up}
H_2(\bl_{\iota}(M);\R) \cong H_2(M;\R) \oplus \mathbb{R} E,
\end{equation}
whose inverse restricts to $\beta_*^{-1}: H_2(M;\R) \ra H_2(\bl_\iota(M);\R)$ and sends $E$ to the homology class of $\CP^1 \times \{\pt\} \subset \CP^{k-1} \times \{\pt\}  \subset  \tilde{B}^{2k}(\epsilon) \times \Sigma \subset \bl_{\iota}(M)$.  

The symplectic form on $\bl_\iota(M)$ is characterized homologically by the property that the symplectic area of $\beta_*^{-1}(A)$ agrees with that of $A$ for any $A \in H_2(M;\R)$, and the symplectic area of $[\CP^1 \times \pt ]$ is $R$.

\item\label{item:bl_def_eq} Suppose $\iota_1:B^{2k}(R_1) \times \Sigma \hooksymp M$ and $\iota_2: B^{2k}(R_2) \times \Sigma \hooksymp M$ are two symplectic embeddings (for some $R_1,R_2 \in \R_{>0}$) such that $\iota_1(\{0\} \times \Sigma)$ and $\iota_2(\{0\} \times \Sigma)$ are isotopic through symplectic submanifolds in $M$. Then the corresponding blowups $(\bl_{\iota_1}(M),\omega_1)$ and $(\bl_{\iota_2}(M),\omega_2)$ are symplectically deformation equivalent in the following sense. For $i = 1,2$, let $(\beta_{i})_*^{-1}: H_2(M) \rightarrow H_2(\bl_{\iota_i}(M))$ denote the map on homology induced by $\beta_i$, and identify $H_2(\bl_{\iota_i} (M))$ with $H_2(M)\oplus \mathbb{R}E_i$ using \eqref{eq:homology_of_blow_up}.

Then there exists a diffeomorphism $\phi: \bl_{\iota_1}(M) \rightarrow \bl_{\iota_2}(M)$ such that:
\begin{enumerate}
\item $\phi_*(E_1) = E_2$;
\item $\phi_* \circ (\beta_1)_*^{-1} = (\beta_2)_{*}^{-1}$;
\item  $\phi^* \omega_2$ is connected to $\omega_1$ by a 1-parameter family of symplectic forms.
\end{enumerate}

\end{enumerate}
\end{proposition}

Part \ref{item:homology_descr} follows a Mayer-Vietoris sequence argument (see \cite[Prop. 2.4]{McDuff_simply_connected}).
Taking the linear dual to \eqref{eq:cohomology_of_blowup} gives
\begin{equation} \label{eq:cohomology_of_blowup}
H^2(\bl_{\iota}(M);\R) \cong H^2(M;\R) \oplus \mathbb{R} e,
\end{equation}
where $\langle e,E \rangle =1$ and $\langle e, A\rangle=0$ for any $A \in H_2(M;\mathbb{R})$.
Under this isomorphism, the symplectic form of $\bl_{\iota}(M;\R)$ lies in class $[\beta^*\omega] + Re$, where
$\langle [\beta^* \omega] ,A\rangle =\langle [\omega],A\rangle$ for any $A\in H_2(M;\mathbb{R})$ and $\langle [\beta^* \omega],E\rangle =0$.

To explain the significance of \ref{item:bl_def_eq}, we observe that the formula for $\bl_\iota(M)$ given in \eqref{eq:blowup_def} depends on the embedding $\iota: B^{2k}(R) \times \Sigma \hooksymp M$ and not just its image.  
Then given two embeddings $\iota_1,\iota_2: B^{2k}(R) \times \Sigma \hooksymp M$ such that $\iota_1(\{0\} \times \Sigma)$ and $\iota_2(\{0\} \times \Sigma)$ are isotopic as symplectic submanifolds in $M$, it is not a priori clear that $\bl_{\iota_1}(M)$ and $\bl_{\iota_2}(M)$ are symplectically deformation equivalent. 
More specifically, given a smooth map $\chi:\Sigma\rightarrow U(k)$ which is not homotopic to a constant map, we can consider the symplectic embedding $\iota_\chi :B^{2k}(R)\times \Sigma \hooksymp M$ defined by $\iota_\chi(a,b):=\iota(\chi(b)(a),b)$ for all $(a,b)\in B^{2k}(R)\times \Sigma$.
Then it is still not immediately clear that $\bl_\iota(M)$ and $\bl_{\iota_\chi}(M)$ are symplectically deformation equivalent, even though $\iota$ and $\iota_{\chi}$ have the same images as embeddings $B^{2k}(R) \times \Sigma \hooksymp M$.
Note that $\iota$ and $\iota_{\chi}$ induce trivializations of the symplectic normal bundle of $\iota(\{0\}\times\Sigma) = \iota_\chi(\{0\}\times\Sigma)$ in $M$ which are (by construction) not homotopic through such trivializations.

We take up the proof of \ref{item:bl_def_eq} in the next section. The idea is that given two symplectic embeddings $\iota_1$ and $\iota_2$, by appropriately shrinking the area parameter $R$, we can deform $\iota_1$ into $\iota_1'$ so that there is some $\chi:\Sigma \rightarrow U(k)$ so that $\iota_2 = \iota_{1\chi}'$. Then we show explicitly $\bl_{\iota_1'}(M)$ and $\bl_{\iota_{1\chi}'}(M)$ are symplectomorphic. The key idea in this step is that the standard symplectic form on $\mathbb{C}^k$ and $\tilde{\mathbb{C}}^k$ are $U(k)$ invariant, so that the identity map $\bl_{\iota_1'}(M)\setminus (\Sigma \times \CP^{k-1}) \rightarrow \bl_{\iota_{1\chi}'}(M) \setminus (\Sigma \times \CP^{k-1})$ can be extended to the entire blowup as a symplectomorphism. 

\begin{remark}
In Section 7.1 of \cite{mcdfuff_salamon_bible} it is explained how to take the symplectic blow up of a symplectic submanifold with nontrivial symplectic normal bundle (See also \cite{McDuff_simply_connected, Guillemin1989}). The conclusion of Proposition \ref{prop:main_properties_of_blowup} is still true for these more general blowups, however the proof of \ref{item:bl_def_eq} becomes more difficult to write down because one needs to keep track of all the choices made during the blowup construction. Since in this paper we only need the simplest case as described in Definition \ref{def:blow_up_R_simple}, we content ourselves to studying blowups in this simpler setting.
\end{remark}

\subsection{Deformation equivalence of the blowup} \label{sec:detailed_blowup}

In this section we prove Part \ref{item:bl_def_eq} of proposition \ref{prop:main_properties_of_blowup}.
We start with 
\begin{proposition} \label{prop:isotopies_give_equivalence_after_blowup}
Let $\iota_t: B^{2k}(R_t)\times \Sigma \hooksymp (M,\omega)$ for $t\in [1,2]$ denote a smooth 1-parameter family of symplectic embeddings, of radius $R_t$ (where $R_t$ also varies smoothly with $t$). Then the blowups $\bl_{\iota_1}(M)$ and $\bl_{\iota_2}(M)$ are deformation equivalent in the sense of Proposition \ref{prop:main_properties_of_blowup}.
\end{proposition}

\begin{proof}
 We observe it suffices to construct a 1-parameter family of symplectic manifolds $(X_t,\omega_t)_{t\in[1,2]}$ so that $(X_1,\omega_1) = \bl_{\iota_1}(M)$ and $(X_2,\omega_2) = \bl_{\iota_2}(M)$.

To be precise, when we say a 1-parameter family of symplectic manifolds, we mean a fiber bundle $X$ over $t\in[1,2]$ equipped with a smooth two form $\omega$. We require the fiber $(X_t,\omega|_{X_t})$ at $t$ to be a symplectic manifold. After choosing a smooth fiber connection, i.e. a splitting of $TX$ into into vertical and horizontal subspaces, on this bundle, this induces a parallel transport map $\varphi_t: X_1\rightarrow X_t$. Thus we may consider instead a fixed smooth manifold $X_1$ but a 1-parameter family of symplectic forms $\varphi_t^*\omega_t$. This shows $(X_1,\omega_1)$ and $(X_2,\omega_2)$ are deformation equivalent in the sense of Proposition \ref{prop:main_properties_of_blowup}.

To do this, consider $M\times [1,2]_t$. We choose $\epsilon$ small enough so that for all $t\in [1,2]$ the map $\iota_t$ extends to a map $\iota_t: B^{2k}(R_t+\epsilon) \times \Sigma \hooksymp M$. We apply the construction in Definition $\ref{def:blow_up_R_simple}$ in families. To be precise we cut out $\iota_t( B^{2k}(R) \times \Sigma) \times [1,2] \subset M\times [1,2]$ and glue back in $\Sigma \times \wt{B}^{2k}(\eps) \times [1,2]$ using the map $(\iota_t \circ ( F_{R_t} \times \id), \id). $
\end{proof}

\begin{proposition} \label{prop:straighten_fiber}
Given two symplectic embeddings $\iota_1, \iota_2: B^{2k}(R)\times \Sigma \hooksymp (M,\omega)$  such that
     \[
     \iota_1|_{\{0\}\times \Sigma } = \iota_2|_{\{0\} \times \Sigma} \quad d\iota_1|_{\{0\}\times \Sigma } = d \iota_2|_{\{0\} \times \Sigma},
     \]
     then after shrinking $R$ appropriately, there exists a one parameter family of symplectic embeddings $\iota_t: B^{2k}(R)\times \Sigma \hooksymp (M,\omega)$ with $t\in [0,1]$ that starts at $\iota_1$ and ends at  $\iota_2$.
\end{proposition}
\begin{proof}
 The proof is modelled after proof of Theorem 3.3.1 in \cite{mcduff2017introduction}.
    First choose a 1-parameter family of smooth embeddings $j_t, t\in [0,1]$  connecting $\iota_1$ and $\iota_2$, so that
    \[
    j_t|_{\{0\}\times \Sigma } = \iota_1|_{\{0\}\times \Sigma }, \quad dj_{t}|_{\{0\}\times \Sigma } = d\iota_1|_{\{0\}\times \Sigma }
    \]
    Then by the relative Poincar\'e lemma, there exists an one parameter family of one forms $\alpha_t$ so that
    \[
    \alpha_t|_{0\times \Sigma} =0, \quad d\alpha_t = j_t^*\omega - (\omega_{std} + \omega|_\Sigma), \quad \alpha_0=\alpha_1 =0
    \]
    Choose $\epsilon$ small enough so that $\omega_{std} + \omega|_\Sigma + sd\alpha_t$ is nondegenerate on $B^{2k}(\epsilon) \times \Sigma$. Then after shrinking $R$ to be small enough,
    it follows from Moser isotopy there exists a 2-parameter family of embeddings
    \[
    \chi_{s,t} : B^{2k}(R) \times \Sigma \rightarrow B^{2k}(\epsilon)\times \Sigma
    \]
    so that \[
    \chi^*_{s,t}(\omega_{std} + \omega|_\Sigma + sd\alpha_t) = \omega_{std} + \omega|_\Sigma, \quad \chi_{s,t}|_{0\times \Sigma} = \id, \quad \chi_{0,t} = \chi_{s,0} = \chi_{s,1} = \id.
    \]
    Our desired family of symplectic embeddings is then $j_t \circ \chi_{1,t}: B^{2k}(R) \times \Sigma \hooksymp M$.
\end{proof}

Given two symplectic embeddings  $\iota_1, \iota_2: B^{2k}(R)\times \Sigma \hooksymp (M,\omega) $ that agree on $\{0\} \times \Sigma$, we now proceed to show $\bl_{\iota_1}(M)$ and $\bl_{\iota_2}(M)$ are deformation equivalent. Once we have this, Part \ref{item:bl_def_eq} of Proposition \ref{prop:main_properties_of_blowup} follow directly from Propositions \ref{prop:straighten_fiber} and \ref{prop:isotopies_give_equivalence_after_blowup}.

Observe the element $d\iota_1 ^{-1}\circ d\iota_2|_{\{0\}\times\Sigma}$ can be naturally viewed as a map $\chi_2: \Sigma \rightarrow Sp(2k)$. Take $R'$ sufficiently small. Consider the symplectic embedding $\iota_{1,\chi_2}:B^{2k}(R') \times \Sigma \hooksymp M$ defined to be 
\[\iota_{1,\chi_2}(a,b) := \iota_1(\chi_2(b)(a),b)\]
for $(a,b) \in B^{2k}(R') \times \Sigma$. Here $R'$ needs to be chosen small enough so that $\chi_2(b)(a) \in B^{2k}(R)$. Then $d\iota_{1,\chi_2}|_{\{0\}\times \Sigma} = d\iota_2|_{\{0\}\times \Sigma}$, hence Propositions \ref{prop:straighten_fiber} and \ref{prop:isotopies_give_equivalence_after_blowup} combine to show $\bl_{\iota_{1,\chi_2}}(M)$ is deformation equivalent to $\bl_{\iota_2}(M)$. On the other hand, there exists a path of $\chi_t:\Sigma \rightarrow Sp(k)$ for $t\in [1,2]$ starting at $\chi_1$ and ending at $\chi_2$ so that $\chi_1$ takes $\Sigma \rightarrow U(k)$. Hence by potentially shrinking $R'$ again we have a one parameter family of symplectic embeddings given by
\[
\iota_{1,\chi_t}: B^{2k}(R') \times \Sigma \hooksymp M, \quad \iota_{1,\chi_t}(a,b): = \iota_1(\chi_t(b)a,b).
\]
Using this 1-parameter family of symplectic embeddings, Proposition \ref{prop:isotopies_give_equivalence_after_blowup} tells us $\bl_{\iota_{1,\chi_2}}(M)$ is deformation equivalent to $\bl_{\iota_{1,\chi_1}}(M)$. The proof of proposition \ref{prop:main_properties_of_blowup} is then concluded by the following proposition.

\begin{proposition}

The symplectic manifolds $\bl_{\iota_1}(M)$ and $\bl_{\iota_{1,\chi_1}}(M)$ are symplectormorphic. Here $\iota_{1,\chi_1}: = \iota_1(\chi_1(b)a,b)$, where $\chi_1: \Sigma \rightarrow U(k)$.
\end{proposition}
\begin{proof}
We first recall elements $g\in U(k)$ act symplectomorphically on $\mathbb{C}^k$ by $z \rightarrow gz \in \mathbb{C}^k$. There is a natural lifted action of $U(k)$ on $(\tilde{\mathbb{C}}^k, \wt{\omega}_R)$ also by symplectomorphisms. To see this, we recall we may write elements of $\tilde{\mathbb{C}}^k$ as pairs $([l],z)$, where $[l]$ is a line through the origin and $z$ is a point on $[l]$. Then $g \cdot ([l],z)= ([gl],gz)$. 

Given $\chi_1: \Sigma \rightarrow U(k)$, the map 
\[f_{\chi_1}:B^{2k}(R+\epsilon) \times \Sigma \rightarrow B^{2k}(R+\epsilon) \times \Sigma, \quad f_{\chi_1} (a,b) := (\chi(b)(a),b)\]
is a symplectomorphism. Likewise, the lifted map \[\wt{f}_{\chi_1}: \tilde{B}^{2k}(\epsilon) \times \Sigma \rightarrow  \wt{B}^{2k} \times \Sigma,\quad
\wt{f}_{\chi_1} (([l],z), b): = \big(([\chi_1(b)l],\chi_1(b)z),b \big).
\]
is also a symplectomorphism.

The map $(F_R, Id): \wt{B}^{2k}(\epsilon)\setminus Z \times \Sigma \rightarrow B^{2k}(R+\epsilon)\setminus B^{2k}(R) \times \Sigma$ is an $U(k)$ equivariant symplectomorphism. In particular
\begin{equation} \label{eq:equivariance}
f_{\chi_1} \circ (F_R, Id) = (F_R, Id) \circ \wt{f}_{\chi_1}
\end{equation}
with domain of the above equation being $\wt{B}^{2k}(\epsilon)\setminus Z \times \Sigma$.

We now recall that $\bl_{\iota_1}(M)$ is constructed by taking 

\begin{align*}
\bl_{\iota_1}(M) := \left(M \setminus \iota_1( B^{2k}(R)\times \Sigma) \sqcup  ( \wt{B}^{2k}(\eps)\times \Sigma) \right) / \sim,
\end{align*}
where the equivalence relation is given by taking $\wt{B}^{2k}(\eps)\setminus Z \times \Sigma$ and gluing it to $M$ via $\iota_1 \circ ( F_R \times \id)$.

Similarly $\bl_{\iota_{1,\chi_1}}(M)$ is constructed by taking the disjoint union 
\begin{align*}
\bl_{\iota_{1,\chi_1}}(M) := \left(M \setminus \iota_{1,\chi_1}( B^{2k}(R)\times \Sigma) \sqcup  (\wt{B}^{2k}(\eps)\times \Sigma) \right) / \sim,
\end{align*}
where the gluing is done by taking $\Sigma \times\wt{B}^{2k}(\eps)\setminus Z$ and gluing it to $M$ via $\iota_{1} \circ f_{\chi_1} \circ ( F_R \times \id)$. 

Then the map from $\bl_{\iota_{\chi_1}}(M)$ to $ \bl_{\iota_1}(M)$ is given by the identity from  $M\setminus \iota_{1,\chi_1}( B^{2k}(R)\times \Sigma) $ to $M \setminus \iota_{1}( B^{2k}(R)\times \Sigma)$. And it is given by $(\wt{f}_{\chi_1},\id)$ from $(\wt{B}^{2k}(\eps) \times \Sigma ) \subset \bl_{\iota_{\chi_1}}(M) $ to $(\wt{B}^{2k}(\eps) \times \Sigma) \subset \bl_{\iota_{1}}(M)$. The map is a symplectomorphism over each of the two glued pieces, and is well-defined by \eqref{eq:equivariance}.

\end{proof}

\section{Stabilizing $J$-holomorphic curves} \label{sec:holomorphic_curves}

    Let $(M,\omega_M)$ be a closed symplectic manifold and $A \in H_2(M)$ an integral homology class. Let $\mathcal{M}_{0,k}(M,A;J)$ denote the space of genus $0$ $J$-holomorphic maps $u: \CP^1 \ra M$ lying in the homology class $A$ with $k$ marked points in the domain.
     Let $\tilde{\mathcal{M}}_{0,k}(M,A;J)$ denote the  moduli space of unparameterized curves given by quotienting by the action of $\op{PSL}(2,\C)$ via biholomorphic reparameterizations. 
     For a curve $v \in \mathcal{M}_{0,k}(M,A;J)$ we denote its image in $\tilde{\mathcal{M}}_{0,k}(M,A;J)$ by $[v]$.

    Let $(N,\omega_N)$ be another closed symplectic manifold. 
    Let $J_M$ and $J_N$ be compatible almost complex structures on $M$ and $N$ respectively, and consider the split almost complex structure $J_M \times J_N$ on $M \times N$. 
    Note that for any fixed $p \in N$ there is an inclusion map
\begin{align*}
\mathcal{M}_{0,k}(M,A;J_M) &\ra \mathcal{M}_{0,k}(M\times N,A\times [\pt];J_M\times J_N)\\v &\mapsto \hat{v}_p,
\end{align*}
where $\hat{v}_p(x) := (v(x),p)$ for any $x \in \CP^1$.
Since the almost complex structure $J_M \times J_N$ is split, every element of $\mathcal{M}_{0,k}(M\times N,A\times [\pt];J_M\times J_N)$ is necessarily constant in the second factor, i.e. of the form $\hat{v}_q$ for some $v \in \calM_{0,k}(M,A;J_M)$ and $q \in N$.

Specializing to the case $k=1$, evaluating at the marked point gives a well-defined smooth map
\[
(\tilde{\ev_M},\tilde{\ev_N}): \tilde{\mathcal{M}}_{0,1}(M\times N,A\times [\pt];J_M\times J_N) \rightarrow M\times N.\]

\begin{theorem} \label{thm:stab_hol_curve}

Assume that all elements of $\mathcal{M}_{0,0}(M,A;J_M)$ are somewhere injective and regular.
Assume also that the almost complex structure $J_N$ is integrable near a fixed point $p \in N$.
Then the moduli space $\tilde{\mathcal{M}}_{0,1}(M\times N,A\times [\pt];J_M\times J_N)$ is regular and $p$ is a regular value of the evaluation map $\tilde{\ev_N}: \tilde{\mathcal{M}}_{0,1}(M\times N,A\times [\pt];J_M\times J_N) \ra N$. In particular, $\ev_N^{-1}(p)$ is a smooth submanifold of $\tilde{\mathcal{M}}_{0,1}(M\times N,A\times [\pt];J_M\times J_N)$.
\end{theorem}

\begin{proof}
Let $U$ be a neighborhood of $p$ in $M$ in which $J_N$ is integrable. For $q \in N$, the pullback tangent bundle of $\hat{v}_q$ splits as $\hat{v}_q^* T(M\times N) = v^*TM\oplus \mathbb{C}^n$ (here $\dim N = 2n$). 
The linearized Cauchy Riemann operator for $\hat{v}_q$ takes the form
\[
D_{\hat{v}_q}: W^{1,\ell}(S^2,\hat{v}_q^*T(M\times N)) \rightarrow L^\ell(S^2,\Omega^{0,1}( \hat{v}_q^*(T(M\times N)))).
\]

Here $W^{1,\ell}(S^2,\hat{v}_q^*T(M\times N))$ denotes the Sobolev completion of the space of vector fields, for some $\ell>2$. 
Since the almost complex structure is split, the linearized operator also splits as $D_M\oplus D_N$, where
\[
D_M:W^{1,\ell}(S^2,\hat{v}_q^*TM )\rightarrow L^\ell(S^2,\Omega^{0,1}( \hat{v}_q^*(TM))),\quad D_N:W^{1,\ell}(S^2, \mathbb{C}^n )\rightarrow L^\ell(S^2,\Omega^{0,1}(\mathbb{C}^n)).
\]
In the above we have used $\mathbb{C}^n$ to denote the trivial $\mathbb{C}^n$ bundle over $S^2$.
$D_M$ is already surjective by assumption. For $D_N$, since $J_N$ is integrable, it suffices to consider its restriction $D_N:W^{1,\ell}(S^2, \mathbb{C} )\rightarrow L^\ell(S^2,\Omega^{0,1}( \mathbb{C}))$. The fact this operator is surjective follows directly from the regularity criterion for line bundles in \cite[Lem. 3.3.1]{mcdfuff_salamon_bible}.

There is a map from $T_qN \rightarrow \ker D_N$ corresponding to moving $q \in N$. This map is bijective by noting the index of $D_N$ is precisely $\textup{dim }N$. This tells us that after requiring the curve $\{\hat{v}_q\}_{q\in N}$ to pass through $M\times \{p\}$, the resulting moduli space is still regular; and $\hat{v}_p$ has the same Fredholm index as $v$. Since the second component of $\hat{v}_p$ is constant, the group of holomorphic reparameterizations only acts on the first factor $v$. After quotienting out by this group, we recover the analogous statements for the unparameterized moduli spaces.
\end{proof}

The proposition above can be recast in terms of a computation of Gromov-Witten invariants of the manifold $M\times N$. In particular it can be seen as an instance of the product formula of Gromov-Witten invariants, for example as explained in \cite{hirschi2022global}.

Let $GW_{g,k,A}^{M}:H_*(M)^{\otimes k} \rightarrow \mathbb{R}$ denote Gromov-Witten invariant counting genus $g$ $J$-holomorphic curves in $M$ in homology class $A$ with $k$ marked points.

\begin{corollary}\label{cor:GW}
Assume $M$, $N$ and $M\times N$ are all semi-positive. Assume all curves counted by $GW_{0,0,A}^{M}$ satisfy the conditions stated in Theorem \ref{thm:stab_hol_curve}. Let $\alpha \in H_*(M\times N)$ denote the homology class $[M]\times [\pt] \in H_*(M\times N)$. Then we have
\[
GW_{0,0,A}^{M} = GW_{0,1,A\times \pt}^{M\times N}(\alpha).
\]
 \qed 
\end{corollary}
\begin{remark}
We expect Theorem \ref{thm:stab_hol_curve} to admit an analogue in the setting of punctured pseudoholomorphic curves, with $M$ a completed symplectic cobordism and $N$ a closed symplectic manifold. After stabilization the punctured holomorphic curves will have punctures asymptotic to Morse-Bott families of Reeb orbits. In this setting one can use Wendl's automatic transversality for punctured curves \cite{wendl2010automatic} (take all punctures to be ``free'' along Morse-Bott families of Reeb orbits) in lieu of the regularity criterion from \cite{mcdfuff_salamon_bible}. 
\end{remark}

\begin{remark}
We observe that the genus zero assumption is essential in the above results. We expect higher genus curves will disappear after a generic perturbation because their indices are negative. To be precise, after stabilizing by a symplectic manifold of dimension $2n$, the Fredholm index of a curve changes by $n(2-2g)$.

The $g=1$ curve scenario is very curious; in the case of closed curves, suppose we start with a rigid $J$-holomorphic curve in $M$. After we stabilize with $N$ we do get an $\dim N$ dimensional family of non-transverse $J$-holomorphic curves, all of which are index zero. After we perturb we would expect a discrete family will survive the perturbation and become transverse curves. The number of such curves we expect to be counted by an obstruction bundle. We give a lower dimensional manifestation of this phenomenon below.
\end{remark}

\begin{example}
Let $M=(S^1 \times S^1, \omega)$ denote the two dimensional torus, with the standard almost complex structure. Then there exists exactly one genus one $J$-holomorphic curve $u: T^2\rightarrow M$, and this curve is transverse (assuming we can vary the domain complex structure). Now stabilize with a symplectic surface of genus $g$, which we write as $(\Sigma_g,\omega)$. 
Now consider $S^1\times S^1 \times \Sigma_g$. With the split almost complex structure we see a 2-dimensional family of non-transverse index zero holomorphic tori parametrized by $g\in \Sigma_g$. We now think the ambient space $S^1\times S^1 \times \Sigma_g$ as $S^1$ times the mapping torus of the identity map on $\Sigma_g$. Now perturb the mapping torus of the identity map to the mapping torus of a small non-degenerate Hamiltonian $H$ (with appropriate choice of almost complex structure).  The curves in the homology class $[T^2] \times [\pt]$ that survive are precisely those corresponding the degree 1 periodic orbits of $X_H$; we have nonzero such counts; and this count is given by Taubes' Gromov invariant, as explained in section 2.6 of \cite{bn}.
\end{example}

\section{h-principles} \label{sec:h_principle}

We now state the main h-principle result we will need to use.

\begin{proposition} \label{prop:h-principle}
     Let $(\Sigma_g(L),\omega_L)$ denote the closed symplectic surface of genus $g$ and area $L$. Consider the symplectic manifold $(\CP^2 \times \Sigma_g(L), \omega_{FS}+ \omega_L)$. Here $\omega_{FS}$ is the Fubini-Study form.
    
    Let $\iota_i: \Sigma_{g}(L)\hookrightarrow \CP^2 \times \Sigma_g(L)$ be a finite collection of disjoint symplectic embeddings, each in the homology class $[\pt] \times [\Sigma_g(L)]$. Then there exists a symplectic isotopy $\phi_t$ with $\phi_0= \id$ and $\phi_1\circ \iota_i (\Sigma_{g}(L)) = \{p_i\} \times \Sigma_g\subset \CP^2\times \Sigma_g(L)$ for a collection of disjoint points $\{p_i\} \in \CP^2$.

\end{proposition}
We first give a proof via an abstract h-principle in the style of \cite{eliashberg2002introduction}; we then use a direct construction using tools from \cite{Buhovsky_Opshtein} to give a proof for the case where $g=0$ and we replace $\CP^2 \times \Sigma_g(L)$ with $B^4(R) \times \Sigma_g(L)$.
\subsection{General h-principle proof}
Let us first recall the (parametric) h principle for symplectic embeddings between symplectic manifolds $(Q^{2k},\omega_Q)$  and $(M^{2n},\omega_M)$.\footnote{Here by ``symplectic embedding'' we mean an embedding $\iota: Q \hookrightarrow M$ such that $\iota^* \omega_M = \omega_Q$. This is called an ``isosymplectic embedding'' in \cite{eliashberg2002introduction}, which also considers the weaker notion of an embedding whose image is a symplectic submanifold.}
Given smooth manifolds $Q^{2k}$ and $M^{2n}$, a smooth map $F: TQ \ra TM$ is a {\em monomorphism} if 
\begin{itemize}
  \item there exists a smooth map $f: Q \ra M$ such that for any $p \in Q$ and $v \in T_pQ$ we have $F(v) \in T_{f(p)}M$
  \item $F|_{T_pQ}: T_pQ \ra T_{f(p)}M$ is an injective linear map for each $p \in Q$.
\end{itemize}
In particular, $F$ gives a vector bundle homomorphism $TQ \ra f^*TM$.
Following \cite{eliashberg2002introduction}, we use the notation $\bs F := f$.
If $Q$ and $M$ are further equipped with symplectic forms $\omega_Q$ and $\omega_M$ respectively, we will say that $F$ is {\em symplectic form preserving} if $F|_{T_pQ}: T_pQ \ra T_{f(p)}M$ pulls back the linear two-form $\omega_M|_{f(p)} \in \Lambda^2(T^*_{f(p)}M)$ to $\omega_Q|_p \in \Lambda^2(T_p^*Q)$ for each $p \in Q$.
\begin{definition}
A {\em formal symplectic embedding} from $Q$ into $M$ is a smooth embedding $f: Q \hookrightarrow M$ which satisfies the cohomological condition $f^*[\omega_M] = [\omega_Q] \in H_2(Q;\R)$, 
together with a smooth family of monomorphisms $F_t: TQ \ra TM$, $t \in [0,1]$, satisfying $\bs F_t = f$ and $F_0 = df$, and such that $F_1$ is symplectic form preserving. 
\end{definition}
\NI Let $\calS(Q,M)$ denote the space of all symplectic embeddings of $Q$ into $M$, and similarly let $\calSformal(Q,M)$ denote the space of all formal symplectic embeddings of $Q$ into $M$. Note that there is a natural inclusion map $\iota: \calS(Q,M) \hookrightarrow \calSformal(Q,M)$ which associates to a symplectic embedding $f: Q \hooksymp M$ the formal symplectic embedding with $F_t = df$ for all $t \in [0,1]$.

\begin{theorem} [{\cite{gromov2013partial}, see also \cite{eliashberg2002introduction}}]\label{thm:h_princ_symp_sub}
Let $(Q,\omega_Q)$ be a closed symplectic manifold and $(M,\omega_M)$ any symplectic manifold such that we have $\dim Q \leq \dim M - 4$.
Then the inclusion map $\iota: \calS(Q,M) \hookrightarrow \calSformal(Q,M)$ is a weak homotopy equivalence.
\end{theorem}
We will be primarily interested in the following special case, which corresponds to injectivity on $\pi_0$.
\begin{corollary}
\label{cor:pi_0_h_princ}
Let $(Q,\omega_Q)$ be a closed symplectic manifold and $(M,\omega_M)$ any symplectic manifold such that we have $\dim Q \leq \dim M - 4$. Let $f_0,f_1: Q \hooksymp M$ be symplectic embeddings which are joined by a smooth isotopy $f_t: Q \hookrightarrow M$, $t \in [0,1]$.
 Suppose further that we have a smooth family of bundle monomorphisms $F_{s,t}: TQ \ra TM$, $s,t \in [0,1]$ such that
 \begin{itemize}
   \item $\bs F_{s,t} = f_t$ for all $s,t \in [0,1]$
   \item $F_{s,i} = df_i$ for all $s \in [0,1]$ and $i = 0,1$
   \item $F_{0,t} = df_t$ for all $t \in [0,1]$
   \item $F_{1,t}$ is symplectic form preserving for all $t \in [0,1]$.
 \end{itemize}
Then $f_0$ and $f_1$ can be joined by a smooth family of symplectic embeddings.
\end{corollary}
\NI Note that under the hypotheses of Corollary~\ref{cor:pi_0_h_princ} we automatically have $f_t^*[\omega_M] = [\omega_Q]$ for all $t \in [0,1]$.

\begin{proposition}\label{prop:formal}
Let $(Q,\omega_Q)$ be a closed symplectic manifold and $M$ any symplectic manifold, and assume that we have $\dim M \geq 2\dim Q+2$.
Let $f_0,f_1: Q \hooksymp M$ be symplectic embeddings which are homotopic as smooth maps. Then $f_0$ and $f_1$ can be joined by a smooth family of symplectic embeddings.
\end{proposition}

\begin{remark}\label{rmk:gen_pos}
Under the assumption $\dim M \geq 2\dim Q + 2$, general position considerations show that any smooth homotopy $f_t: Q \ra M$, $t \in [0,1]$, becomes a smooth isotopy after a small perturbation.
\end{remark}

Let $E$ denote the fiber bundle over $Q \times [0,1]^2$ whose fiber $E_{q,s,t}$ over $(q,s,t) \in Q \times [0,1]^2$ is 
the space $\hom(T_qQ,T_{f_t(q)}M)$ of linear maps from $T_qQ$ to $T_{f_t(q)}M$.
Similarly, let $E' \subset E$ denote the subbundle whose fiber $E'_{q,s,t}$ over $(q,s,t)$ is the space
$\hom_\omega(T_qQ,T_{f_t(q)}M) \subset \hom(T_qQ,T_{f_t(q)}M)$ of symplectic form preserving linear maps.
Note that $E_{q,s,t}$ is diffeomorphic to the space $\fr_{2k}(\R^{2n})$ of $(2k)$-frames in $\R^{2n}$, while $E'_{q,s,t}$ is diffeomorphic to the space $\fr_{2k}^\omega(\R^{2n})$ of symplectic $(2k)$-frames in $\R^{2n}$.
One can also check that $\fr_{2k}(\R^{2n})$ is homotopy equivalent to the real Stiefel manifold $O(2n)/O(2n-2k)$ and $\fr^\omega_{2k}(\R^{2n})$ is homotopy equivalent to the  complex Stiefel manifold $U(n)/U(n-k)$.

\begin{lemma}\label{lem:formal_iso_stiefel}
Let $(Q^{2k},\omega_Q)$ and $(M^{2n},\omega_M)$ be symplectic manifolds, and let 
$f_0,f_1: Q \hooksymp M$ be symplectic embeddings which are smoothly isotopic.
If the pair $(\fr_{2k}(\R^{2n}),\fr_{2k}^\omega(\R^{2n}))$ is $(2k+1)$-connected, then $f_0$ and $f_1$ can be joined by a smooth family of formal symplectic embeddings.
\end{lemma}
\begin{proof}
This follows by a standard obstruction theory argument. 
It suffices to find a section $\sigma: Q \times [0,1]^2 \ra E$ which agrees with $F$ over $Q \times [0,1] \times \{0,1\}$ and $Q \times \{0\} \times [0,1]$ and has image in $E'$ over $Q \times \{1\} \times [0,1]$.
Fixing a CW structure on $Q$, we construct $\sigma$ cell-by-cell.
Firstly, for each $0$-cell $C_0$ of $Q$ we pick compatible trivializations of $E$ and $E'$ over $C_0 \times [0,1]^2$. 
This identifies the restriction of $E$ to $C_0 \times [0,1]^2$ with $\fr_{2k}(\R^{2n}) \times [0,1]$, such that $E'|_{C_0 \times [0,1]^2}$ corresponds to $\fr_{2k}^\omega(\R^{2n}) \times [0,1]^2$.
Then pick any extension of $\sigma$ over $C_0 \times [0,1]^2 \cong [0,1]^2$, which exists since $\pi_1(\fr_{2k}(\R^{2n}),\fr_{2k}^\omega(\R^{2n})) = 0$ by assumption.

Now suppose by induction that $\sigma$ is already defined over all products of open $j$-cells in $Q$ with $[0,1]^2$ as $j$ ranges over $0,\dots,i-1$ for some $i \leq 2k$, and let $C$ be an open $i$-cell of $Q$. It suffices to extend $\sigma$ over $C \times [0,1]^2$.
By a deformation retract argument, we can assume that $\sigma$ is already defined in an open  neighborhood of $\bdy C \times [0,1]^2$. Let $C'$ be a slight shrinkening of $C$ which is diffeomorphic to a closed $i$-cell,
so that $\sigma$ is already defined near $\bdy C' \times [0,1]$ and $C' \times \{0\}$.
Then we can extend $\sigma$ over $C' \times [0,1]^2$ as long as $\pi_{i+1}(\fr_{2k}(\R^{2n}),\fr_{2k}^\omega(\R^{2n})) = 0$, which holds by assumption. 
\end{proof}

\begin{lemma}\label{lem:frame_mfd_conn}
The real frame manifold $\fr_{2k}(\R^{2n})$ is $(2n-2k-1)$-connected and the symplectic frame manifold $\fr_k(\C^n)$ is $(2n-2k)$-connected. In particular, the pair $(\fr_{2k}(\R^{2n}),\fr^\omega_{2k}(\R^{2n}))$ is $(2n-2k-1)$-connected.
\end{lemma}
\begin{proof}
 The first claim follows by an easy induction using fiber sequence $\R^{i-1} \times \fr_{i-1}(\R^{j-1}) \ra \fr_i(\R^j) \ra \R^{j}\setminus \{\vec{0}\}$, and the second claim similarly follows using the fiber sequence $\R^{2j-1}\times \fr_{2i-2}^\omega(\R^{2j-2}) \ra \fr_{2i}^\omega(\R^{2j}) \ra \R^{2j} \setminus \{\vec{0}\}$.
\end{proof}

\begin{proof}[Proof of Proposition \ref{prop:formal}]
Put $\dim M = 2n$ and $\dim Q = 2k$. According to Remark~\ref{rmk:gen_pos}, $f_0$ and $f_1$ can be joined by a smooth isotopy $f_t: Q \hookrightarrow M$, $t \in [0,1]$. Then by Lemma~\ref{lem:formal_iso_stiefel} it suffices to check that $(\fr_{2k}(\R^{2n}),\fr^\omega_{2k}(\R^{2n}))$ is $(2k+1)$-connected, and, since by assumption we have $2n-2k-1 \geq 2k+1$, this follows by Lemma~\ref{lem:frame_mfd_conn}.
\end{proof}

Proposition \ref{prop:h-principle} then follows from Proposition \ref{prop:formal} and the isotopy extension theorem in \cite{auroux_asymptotic}.

\subsection{A direct construction}

In this subsection we give an alternative proof of proposition \ref{prop:h-principle} in the special case that $\Sigma_g(L)$ has genus zero, which we write as $S^2(L)$. We use the following proposition found in the appendix of  \cite{Buhovsky_Opshtein}.

\begin{proposition}[{\cite[Lem. A.5]{Buhovsky_Opshtein}}] \label{prop:A.5}

Let $W$ denote an open subset of $\mathbb{C}^n$, $n\geq 3$,

which is diffeomorphic to a ball. Let $v_1, v_2:D(r) \rightarrow W$ be symplectic embeddings of a symplectic closed disc of area $r$ that coincide on a neighborhood of the boundary. Then there is a compactly supported Hamiltonian function $H:W\times[0,1] \rightarrow W$ such that near a neighborhood of $v_i(\partial D)$ we have $H(\cdot,t) =0$, and for $\phi_1$ the time $1$ flow of $H$ we have $\phi_1 \circ v_1 = v_2$.
\end{proposition}

By contrast to Theorem~\ref{thm:h_princ_symp_sub}, which is based on the technique of holonomic approximation (see \cite[\S3]{eliashberg2002introduction}), the proof of \ref{prop:A.5} is very hands on and describes the Hamiltonian needed to realize this isotopy fairly explicitly.

\begin{remark}
The statement in \cite{Buhovsky_Opshtein} it requires $W$ to be symplectomorphic to the standard symplectic ball, but upon inspecting the proof this assumption is not needed.
\end{remark}

\begin{proof}[Proof of genus zero version of Proposition \ref{prop:h-principle}]
We consider first $f:S^2(L) \hooksymp B^4(R) \times S^2(L)$ a symplectic embedding in the homology class $[\pt] \times [S^2]$. We shall first explain how to symplectically isotope this to the zero section of the form $0 \times S^2(L)$.

    Let $p \in S^2$ be a point and $D_\epsilon(p) \subset S^2$ be a disc of radius $\epsilon$ around $p$. We can, via a symplectic isotopy, assume that $f$ maps $D_\epsilon(p)$ to $0 \times D_\epsilon(p) \subset B^4(R) \times D_\epsilon(p)$. 
    
    Next consider $F_p$, the fiber of the trivial bundle $B^4(R) \times S^2 \rightarrow S^2(L)$ centered at $p$. Then $F_p$ is diffeomorphic to $B^4(R)$ and has homological intersection 1 with the image of $f$. We deform $F_p$ by the Whitney trick so that its geometric intersection with the image of $f$ is one and precisely at $p$. During this deformation we assume the boundary of $F_p$ remains on $\partial B^4(R) \times S^2(L)$. We denote the resulting submanifold by $\tilde{F}_p$.

    Let $\tilde{F}_{\epsilon'}(p)$ denote a thickening of size $\epsilon' \ll \epsilon$ of $\tilde{F}_p$, so that $\tilde{F}_{\epsilon'}(p) \cap 0\times S^2(L) = 0 \times D_{\epsilon'}(p)$. Then consider
    \[
    W_{\epsilon'} : =  B^4(R) \setminus \tilde{F}_{\epsilon'}(p).
    \]
    Then $W_{\epsilon'}$ is diffeomorphic to the six ball (but not necessarily symplectomorphic). Then consider $f: S^2(L) \setminus D_{2\epsilon'}(p) \rightarrow W_{\epsilon'}$. By A.5 there is a symplectic isotopy that takes $f|_{ S^2(L) \setminus D_{2\epsilon'}(p)}$ to $0 \times S^2 \setminus D_{2\epsilon'}(p) \subset B^4(R) \times S^2(L)$ that fixes a neighborhood of the boundary. Since this isotopy fixes a neighborhood of the boundary we can extend it to all of $S^2(L)$, and obtain an isotopy from $f|_{S^2(L)}$ to $0\times S^2$. This symplectic isotopy may have self intersections, in particular interior intersections with $0\times D_{2\epsilon'}(p)$, but we can perturb them away generically. Since $S^2(L)$ is simply connected, we can upgrade this isotopy of symplectic embeddings to an ambient Hamiltonian isotopy, supported away from $\partial B^4(R) \times S^2(L)$.

To prove the result for a collection of disjointly symplectically embedded $S^2(L)$s, we observe we can perturb the isotopies of symplectic embeddings so that the isotopies of embedded symplectic spheres don't intersect each other. For spheres all symplectic isotopies are Hamiltonian, so we can extend the family of symplectic embeddings to an ambient Hamiltonian flow.
\end{proof}

\section{Proof of main theorem}\label{sec:pf_main_thm}
In this section we assemble the previous ingredients to prove our main theorem. As preparation, we first give a review of symplectic ball packings in dimension four.

\subsection{Review of ball packings in dimension four}

Given a collection of points $p_1,\dots,p_k \in \CP^2(R)$, we denote the corresponding complex blowup at these points by $\bl_{\{p_i\}}\CP^2$. Let $A \in H_2(\bl_{\{p_i\}}\CP^2)$ denote the line class and $E_1,\dots,E_k \in H_2(\bl_{\{p_i\}}(\CP^2)$ the homology classes of the corresponding exceptional divisors.

Let $-K:=3A-\sum_i E_i \in H_2(\bl_{\{p_i\}}\CP^2)$ denote the anticanonical class.  We define $\mathcal{E}_K\subset H_2(\bl_{\{p_i\}} (\CP^2)$ to be the set of homology classes $F$ satisfying $K\cdot F=-1$ and $F \cdot F = -1$ that can be represented by a smoothly embedded sphere in $\bl_{\{p_i\}} (\CP^2)$. 

\begin{definition} [{\cite[Prop 3.2]{McDuff_Hofer_conjecture} or \cite[Def. 1.3.A]{mcduff1994symplectic}}]
    A vector $(d,m_1,...,m_k) \in \mathbb{Z}_+ \times \mathbb{Z}^k_{\geq0}$ is called {\em exceptional} if the homology class $dA - \sum_i^k m_i E_i$ belongs to $\mathcal{E}_K$.
\end{definition}

\begin{proposition} [{\cite[Prop 3.2]{McDuff_Hofer_conjecture} or \cite[Thm 1.3.C]{mcduff1994symplectic}}] \label{prop:exceptional_vector_obstruct} 
Suppose that a symplectic packing of $k$ balls of area $R_i$ in $\CP^2(R)$ exists, then for each exceptional vector $(d,m_1,\cdots,m_k)$ we have the inequality
\[
\sum_i^k m_i R_i < d  R.
\]
In fact, the set of obstructions coming from exceptional vectors is equivalent to the combinatorial obstructions listed in \ref{thm:combinatorial_packing_obstructions}, and hence form a system of sharp ball packing obstructions.
\end{proposition}

The proof of the obstructions proceed as follows. Suppose such a ballpacking $\iota: \bigsqcup_{i=1}^k B^4(R_i) \hooksymp CP^2(R)$ exists, this induces a symplectic blow up of radius $R_i$ around each embedded ball. We denote this blow up by $\bl_{\{\iota_i\}}(\CP ^2(R))$.The blown up symplectic form on $\bl_{\{\iota_i\}}(\CP ^2(R))$ has first Chern class Poincare dual to $K$ and cohomology class $R\alpha + \sum_i R_i e_i$. This means each exceptional rational curve can be represented by a rational symplectic sphere (Lemma 3.5 \cite{li2001uniqueness}, see also Theorem C in \cite{li_li_2002}), and knowing the area of a symplectic sphere is positive gives the constraints above.
Through the work of McDuff \cite[Prop. 3.2]{McDuff_Hofer_conjecture}, the set of obstructions coming from exceptional vectors are in fact equivalent to the obstructions presented in Theorem \ref{thm:combinatorial_packing_obstructions}.

We now record some properties of the exceptional curves from which the embedding obstruction arises. The description below follows from the exposition in \cite[Prop. 2.3.A]{mcduff1994symplectic}.

Let $C$ denote an exceptional symplectic sphere on $\bl_{\{\iota_i\}}(\CP ^2(R))$ in the homology class given by the exceptional vector $(d,m_1,...,m_k)$, and $J$ a generic compatible almost complex structure chosen so that $C$ is $J$-holomorphic. Then as explained in \cite{mcduff1994symplectic}, the curve $C$ is rigid, regular,
 and unique in its homology class. Hence it follows that
\begin{lemma}\label{lem:exceptional_curves_GW}
Let $(d,m_1,\dots, m_k)$ be an exceptional vector and $[C]$ the corresponding homology class in $\bl_{\{\iota_i\}}(\CP ^2(R))$, then
\[GW^{\bl_{\{\iota_i\}}(\CP ^2(R))}_{0,0,[C]} = 1.\] 
\end{lemma}

\subsection{Proof of main theorem}
For $\epsilon>0$ sufficiently small, consider the symplectic embedding $j:\bigsqcup_{i=1}^kB^4(\epsilon) \hooksymp \CP^2(R)$ with the centers of the balls given by $\{p_i\}\subset \CP^2(R)$. Let $\bl_{\{j_i\}}(\CP^2(R))$ denote the resulting blow up. Recall second homology group of $\bl_{\{j_i\}}(\CP^2(R))$  as a vector space over $\mathbb{R}$ is given by
\[
H_2(\bl_{\{j_i\}}(\CP^2(R))) = \mathbb{R}A \oplus \bigoplus_{i=1}^k \mathbb{R}E_i
\]
Consider the embedding $\hat{j}: \bigsqcup_{i=1}^k B^4(\epsilon) \times \Sigma_g(L) \hooksymp \CP^2(R) \times \Sigma_g(L)$, which is given by the extension of $j$ by identity in the second component. Consider the space $(\bl_{\{j_i\}}(\CP^2(R)))\times \Sigma_g(L)$. This space is the same as the size $\epsilon$ blowup of $\CP^2(R) \times \Sigma_g(L)$ along $\{p_i\}\times \Sigma_g(L)$, which we denote by $\bl_{\{\hat{j}_i\}}(\CP^2 \times \Sigma_g(L))$. The second cohomology of this space is given by
\[
H^2(\bl_{\{\hat{j}_i\}}(\CP^2 \times \Sigma_g(L))) = \mathbb{R}\alpha \bigoplus \bigoplus_{i=1}^k \mathbb{R}e_i \bigoplus \mathbb{R} \omega_{L}.
\]
Here $\alpha$ is the cohomology class of the pullback of the Fubini-Study form on $\CP^2(R)$ to the blowup $\bl_{\{j_i\}}(\CP^2(R))$ (see Proposition \ref{prop:main_properties_of_blowup}). We normalize it so that $\langle [\CP^1], \alpha \rangle =1$. Similarly let $[\Sigma_g(L)]$ denote the fundamental class of $\Sigma_g(L)$ under inclusion, we have the normalization $\langle [\Sigma_g(L)], \omega_L \rangle  =L$. 
The resulting blown up symplectic form on this space has cohomology class given by
\[
R[\alpha] + \sum \epsilon [e_i] + [\omega_L].
\]
\begin{proof}[Proof of Theorem \ref{thm:main_thm}]
It suffices to show for any exceptional vector gives rise to the same obstruction as Proposition \ref{prop:exceptional_vector_obstruct}.

Suppose we have stabilized embeddings $\iota: \bigsqcup_{i=1}^k B^4(R_i) \times \Sigma_{g}(L)\hooksymp B^4(R) \times \Sigma_g(L)$. We can view this as collection of symplectic embeddings into $\CP^2(R) \times \Sigma_{g}(L)$. We form the size $R_i$ blowup and denote the resulting space by $\bl_{\{\iota_i\}}(\CP^2(R) \times \Sigma_g(L))$.

\textbf{Claim}: The symplectic manifolds $\bl_{\{\hat{j}_i\}}(\CP^2(R) \times \Sigma_g(L))$ and $\bl_{\{\iota_i\}}(\CP^2(R) \times \Sigma_g(L))$ are deformation equivalent. The symplectic form on $\bl_{\{\iota_i\}}(\CP^2(R) \times \Sigma_g(L))$ lives in the cohomology class $
R[\alpha] +\sum R_i [e_i] + [\omega_L].$

The first part of the claim follows from the second part of Proposition \ref{prop:main_properties_of_blowup}, and the second claim follows from the first part of Proposition \ref{prop:main_properties_of_blowup}.

To finish the proof,
consider an exceptional curve $C$ corresponding to the exceptional vector $(d,m_1,...,m_l)$ in $\bl_{\{j_i\}}(\CP^2(R))$. By lemma \ref{lem:exceptional_curves_GW}, $GW_{0,0,[C]}^{\bl_{\{j_i\}}(\CP^2(R))} =1$, and hence Corollary \ref{cor:GW} implies \[GW_{0,1,[C]\times [\pt]}^{\bl_{\{\hat{j}_i\}}(\CP^2(R) \times \Sigma_g(L))}([\bl_{\{j_i\}}(\CP^2(R)]\times [\pt]) = \pm 1.\]
Since Gromov Witten invariants are preserved by deformations of the symplectic form, this means the same Gromov Witten invariant is nonzero for $\bl_{\{\iota_i\}}(\CP^2(R)\times \Sigma_g(L))$. Measuring this Gromov-Witten invariant against the cohomology class 
\[
R[\alpha] + \sum R_i [e_i] + [\omega_L]
\]
and noting the fact the integral of a symplectic form on any $J$-holomorphic curves is positive, we have
\[
    \sum m_iR_i < dR.
\]
\end{proof}

\begin{remark}
\label{rmk:stab_in_higher_dimensions}
Let $N$ denote any closed symplectic manifold. 
Assume there is a version of Gromov-Witten invariants for which the conclusion of Corollary \ref{cor:GW}
holds without any assumption on semi-positivity.
Consider symplectic embeddings $F:\bigsqcup\limits_{i=1}^k B^4(R_i)\times N \hooksymp \nt B^4(R) \times N$ such that $F(\bigsqcup\limits_{i=1}^k 0 \times N) \subset \nt B^4(R) \times N$ is symplectically isotopic to the symplectic submanifold $\bigsqcup\limits_{i=1}^k \{p_i\} \times N$. Then it follows from the techniques of this paper that such symplectic embedding is possible if and only if there is a symplectic embedding $\bigsqcup\limits_{i=1}^k B^4(R_i)\hooksymp \nt B^4(R)$.

To prove this, the only difference from the proof of Theorem \ref{thm:main_thm} is the lack of an h-principle, but this is provided by the symplectic isotopy of the blowup locus stated here as an assumption. The existence of holomorphic curves is provided by Gromov-Witten invariants.

\end{remark}

\section{Stabilized embeddings for toric domains} \label{sec:stabilized_toric}
Consider the moment map $\mu: \mathbb{C}^2 \rightarrow \mathbb{R}^2$ given by
\[
(z_1,z_2) \rightarrow ( \pi |z_1|^2, \pi |z_2|^2).
\]
Let $\Omega$ denote a region in $\mathbb{R}^2$, we define a \emph{toric domain} $X_\Omega$ to be the pre-image of $\Omega$ under $\mu$.
\begin{definition} [\cite{DanCG_Concave_into_Convex}]\label{def:concave_tor}
A toric domain $X_\Omega$ is concave if $\Omega$ is the region in the first quadrant under the graph of a convex function $f: [0,a] \rightarrow [0,b]$, with $a,b>0$ and $f(0)=b$ and $f(a)=0$.
\end{definition}

\begin{definition} [\cite{DanCG_Concave_into_Convex}]
\label{def:convex_tor}
    A toric domain $X_\Omega$ is convex if $\Omega$ is a closed region in the first quadrant bounded by the coordinate axes and a convex curve connecting $(a,0)$ and $(0,b)$, for $a,b >0$.
\end{definition}

\begin{definition} [\cite{DanCG_Concave_into_Convex}]
We say a toric domain $X_\Omega$ is rational, if $\Omega$ has upper boundary that is piecewise linear with rational slopes.
\end{definition}

Given a rational convex or concave toric domain, we can associate to it a finite sequence of numbers called a \emph{weight sequence}. For details of this construction we refer the reader to \cite{DanCG_Concave_into_Convex}. The point is that for a concave toric domain we can associate to it a sequence of real numbers potentially with repetitions $w(\Omega) = (a_1,a_2,...)$. Associated to this sequence of numbers is a collection of symplectic balls, which we write as
\[
B^4(\Omega) = \bigsqcup_i B^4(a_i).
\]

For convex toric domains, the weight sequence $w(\Omega)$ takes the form $(b; b_1,b_2...)$. The first number $b$ is called the head, and the remaining un-ordered sequence of numbers is called the negative weight sequence. Associated to the negative weight sequence is a collection of symplectic balls of the form
\[
\hat{B}(\Omega) = \bigsqcup_i B^4(b_i).
\]

\begin{theorem}[{\cite[Thm 2.1]{DanCG_Concave_into_Convex}}]
Let $X_{\Omega_1}$ be a rational concave toric domain, and let $X_{\Omega_2}$ be a rational convex toric domain. Using the notation for weight sequences as above, 
  there exists a symplectic embedding
\[
\nt(X_{\Omega_1}) \hooksymp \nt(X_{\Omega_2})
\]
if and only if there exists a symplectic embedding
\[
\nt(B(\Omega_1) )\bigsqcup \nt (\hat{B}(\Omega_2)) \hooksymp \nt(B(b_1)).
\]

\end{theorem}

\begin{proof}[Proof of Theorem \ref{thm:stabilized_toric}]
It suffices to prove the only if direction.

\textbf{Step 1}. We first prove it for the case where both $X_{\Omega_1}$ and $X_{\Omega_2}$ are rational. Suppose we have a symplectic embedding
\[
\nt(X_{\Omega_1}) \times \Sigma_g(L) \hooksymp \nt(X_{\Omega_2}) \times \Sigma_g(L)\]

If we decompose both $X_{\Omega_1}$ and $X_{\Omega_2}$ with the weight sequence, we obtain stabilized (open) ball packings
\[
\nt(B(\Omega_1)) \times \Sigma_g(L) \bigsqcup \nt (\hat{B}(\Omega_2)) \times \Sigma_g(L) \hooksymp \nt(B(b_1)) \times \Sigma_g(L).
\]
Then for any $0<\lambda<1$, there exists packings 
\[
(B(\lambda \Omega_1)) \times \Sigma_g(L) \bigsqcup (\hat{B}(\lambda \Omega_2)) \times \Sigma_g(L) \hooksymp \nt(B(b_1)) \times \Sigma_g(L).
\]
Here we use $\lambda \Omega_i$ to denote the toric domain rescaled by $\lambda$. Then by 
Theorem \ref{thm:main_thm} then implies there exists symplectic embeddings 
\[
B( \lambda \Omega_1) \bigsqcup \hat{B}( \lambda \Omega_2)) \hooksymp B(b_1).
\]
for all $0<\lambda<1$. As explained in \cite{McDuff_2009}, this means $
\nt(B(\Omega_1)) \bigsqcup \nt(\hat{B}( \Omega_2)) )\hooksymp B(b_1).
$
which implies a symplectic embedding
\[
\nt(X_{\Omega_1}) \hooksymp \nt(X_{\Omega_2}).\]

\textbf{Step 2}. To deduce the general case, suppose $X_{\Omega_1}$ is any concave toric domain and $X_{\Omega_2}$ is any convex toric domain, with an embedding of the form
\[
\nt(B(\Omega_1) \times \Sigma_g(L) \bigsqcup \nt (\hat{B}(\Omega_2)) \times \Sigma_g(L) \hooksymp \nt(B(b_1)) \times \Sigma_g(L).
\]
The for any rational concave toric domain $\Omega' \subset \Omega_1$, and for any rational convex toric domain $\Omega'' \supset \Omega_2$, the previous set implies we have the symplectic embedding
\[
\nt(X_{\Omega'})\hooksymp \nt(X_{\Omega''}).
\]
Then the conclusion of our theorem follows as in the proof of \cite[Thm. 1.2]{DanCG_Concave_into_Convex}.
\end{proof}

\section{Higher dimensional ball packings} \label{section:higher_dim_packing}

To motivate the ball packing problem we first give a review of the proof of Gromov's two ball theorem.

\begin{proof}[Proof of Theorem~\ref{thm:two_ball}]

Given a symplectic embedding $\iota: B^{2n}(R_1) \bigsqcup B^{2n}(R_2) \hooksymp \nt B^{2n}(R)$. We compactify the target $\nt B^{2n}(R)$ to $\CP^n(R)$. Consider the corresponding blowup, $\bl_{\{\iota_i\}}\CP^n(R)$. 
Recall that we have a natural identification $H_2(\bl_{\{\iota_i\}}\CP^n(R)) \cong \Z\langle A,E_1,E_2\rangle$,
where $A\in H_2(\bl_{\{\iota_i\}}\CP^n(R))$ denotes the line class and $E_1,E_2 \in H_2(\bl_{\{p_i,R_i\}}\CP^n(R))$ denote the homology classes of lines in the resulting exceptional divisors.
Recall that the Gromov--Witten invariant $GW^{\bl_{\{\iota_i\}}\CP^n(R)}_{0,0,A-E_1-E_2}$ is nonzero (essentially since there is a unique line through any two points in projective space). \footnote{To be more precise, since this blowup is not necessarily semi-positive, some additional technology is needed to make sure this is well defined. In the case of the Gromov two ball theorem the proof can be carried out entirely without blow up, however to be compatible with the formal calculations we provide below for Conjecture \ref{conj:high_d_ball_packing}, we phrase everything in terms of the blown up projective space.}
In particular, for any tame almost complex structure $J$ on $\bl_{\{\iota_i\}}\CP^n(R)$, the compactified moduli space of stable maps
$\ovl{\calM}_{0,0}(\bl_{\{\iota_i\}}\CP^n(R),A-E_1-E_2; J)$ is nonempty, and the energy of any curve $C \in \ovl{\calM}_{0,0}(\bl_{\{\iota_i\}}\CP^n(R),A-E_1-E_2; J)$ satisfies
$0 \leq \En(C) = R - R_1 - R_2$.
\end{proof}

Now let $\iota$ be a symplectic embedding of the form given in Conjecture \ref{conj:high_d_ball_packing}, and let $\bl_{\{\iota_i\}}\CP^n(R)$ denote the corresponding symplectic blowup of $\CP^n$ at $k$ points.
For nonnegative integers $d,m_1,\dots,m_k \in \Z_{\geq 0}$, put
\begin{align}
\En_\veca(d;m_1,\dots,m_k) := dR - \sum_{i=1}^k m_iR_i,
  \end{align}
which represents the energy of a hypothetical pseudoholomorphic curve in $\bl_{\{\iota_i\}}\CP^n(R)$ in homology class $[C] := dA - \sum_{i=1}^k m_iE_i \in H_2(\bl_{\{\iota_i\}}\CP^n(R))$.
Similarly, for $g \in \Z_{\geq 0}$ put
\begin{align}
\ind_g(d;m_1,\dots,m_k) := (n-3)(2-2g) + 2(n+1)d - 2(n-1)\sum_{i=1}^k m_i,
\end{align}
which represents the (real) Fredholm index of a hypothetical curve in $\bl_{\{\iota_i\}}\CP^n(R)$ of genus $g$ in homology class $[C]$ (without any further constraints imposed).

\begin{lemma}\label{lem:beyond_two_ball}
Given $n \geq 3$, $d \geq 1$, and $\veca = (a_1,\dots,a_k) \in \R_{>0}^k$ for some $k \in \Z_{\geq 1}$, suppose that we have:
\begin{enumerate}[label=(\alph*)]
  \item \label{item:eq} $\ind_g(d;m_1,\dots,m_k) \geq 0$
  \item\label{item:one_ball} $m_i \leq d$ for all $1 \leq i \leq k$
  \item\label{item:two_ball} $R_i + R_j \leq R$ for all $1 \leq i < j \leq k$.
\end{enumerate}
Then we have $\En_\veca(d;m_1,\dots,m_k) \geq 0$.
\end{lemma}
\begin{proof}
Assume without loss of generality that we have $a_1 \geq \cdots \geq a_k$.
Using \ref{item:eq} we have
\begin{align*}
\sum_{i=1}^k m_i \leq \frac{n-3 + (n+1)d}{n-1} \leq 2d,
\end{align*}
and hence using \ref{item:one_ball} and \ref{item:two_ball} we have 
\begin{align*}
\En_\veca(d;m_1,\dots,m_k) \geq dR - dR_1 - dR_2 \geq 0.
\end{align*}
\end{proof}

Lemma~\ref{lem:beyond_two_ball} says roughly that any pseudoholomorphic curve bound obtained as in the above proof of Theorem~\ref{thm:two_ball} cannot improve upon the bound coming from Gromov's two ball theorem (applied to each pair $i,j$).
Indeed, to see this, consider a homology class of the form $[C] := dA - \sum_{i=1}^k m_i E_i \in H_2(\bl_{\{\iota_i\}}\CP^n(R))$ such that $\ind_g(d;m_1,\dots,m_k) \ge 0$.
Suppose we knew that for any symplectic embedding $\iota: \bigsqcup\limits_{i=1}^k B^{2n}(R_i) \hooksymp \CP^n(R)$ there exists a tame almost complex structure $J$ on $\bl_{\{\iota_i\}}\CP^n(R)$ so that the compactified moduli space $\ovl{\calM}_{0,0}(\bl_{\{\iota_i\}}\CP^n(R),[C]; J)$ is nonempty.
For example, this holds if the absolute Gromov--Witten invariant $GW^{\bl_{\{\iota_i\}}\CP^n(R)}_{0,0,[C]}$ is nonzero, or more generally if there is a nonvanishing Gromov--Witten invariant relative to the exceptional divisors in the homology class $[C]$.

In this case, given any ball packing $\bigsqcup\limits_{i=1}^k B^{2n}(R_i') \hooksymp \CP^n(R)$, by energy considerations as in the proof of Theorem~\ref{thm:two_ball} we must have $\sum_{i=1}^k m_iR_i' \leq dR$. In particular, since we can always find such a  ball packing with $R_1' = R-\eps$ for $\eps >0$ sufficiently small, we must have $(1-\eps)m_1 \leq d$, and hence $m_1 \leq d$. Similarly, we must have $m_2,\dots,m_k \leq d$.
Note that \ref{item:two_ball} in Lemma~\ref{lem:beyond_two_ball} holds by Gromov's two ball theorem.
Then Lemma~\ref{lem:beyond_two_ball} says that the energy of any curve $C \in \ovl{\calM}_{0,0}(\bl_{\{\iota_i\}}\CP^n(R),[C]; J)$ already satisfies $\En(C) \geq 0$, i.e. we do not obtain any further obstruction.

The above discussion also easily extends to punctured curves say in the symplectic completion of the complementary symplectic cobordism $X := B^{2n}(R) \setminus \bigsqcup\limits_{i=1}^k \iota(\nt B^{2n}(R_i))$ of a ball packing $\iota: \bigsqcup\limits_{i=1}^k B^{2n}(R_i) \hooksymp B^{2n}(R)$.
Namely, put $\bdy^+ X := \bdy B^{2n}(1)$ and $\bdy^- X := \bigsqcup\limits_{i=1}^k \iota(\bdy B^{2n}(R_i))$,
and let $\wh{X} = X \cup (\bdy^+X \times \R_{\geq 0}) \cup (\bdy^-X \times \R_{\leq 0})$ denote the symplectic completion of $X$.\footnote{Strictly speaking the contact boundaries $\bdy^\pm X$ have Morse--Bott Reeb dynamics, but if desired we can slightly perturb each round ball into an irrational ellipsoid, in which case the same discussion carries over with minor adjustments.}

Given a finite energy punctured curve $C$ of genus $g$ in $\wh{X}$, we can associate to each puncture the multiplicity of the corresponding Reeb orbit. 
Let $d$ denote the summed multiplicity of all punctures which are asymptotic to Reeb orbits in $\bdy^+X$. Similarly, for $i = 1,\dots,k$, let $m_i$ denote the summed multiplicity of all punctures which are asymptotic to Reeb orbits in the $i$th boundary component of $\bdy^-X$.
Then it is easy to check that the Fredholm index and energy of $C$ satisfy $\ind(C) \leq \ind_g(d;m_1,\dots,m_k)$ and $\En(C) = \En_\veca(d;m_1,\dots,m_k)$.
The association of values $d,m_1,\dots,m_k$ to a curve $C$ naturally extends to elements of the SFT compactification by pseudoholomorphic buildings.
The upshot is that given any mechanism for producing punctured curves in $\wh{X}$ with nonnegative index which insensitive to the sizes $R_1,\dots,R_k$, the corresponding inequality given by applying Stokes' theorem cannot beat Gromov's two ball theorem.

Specializing to balls of equal size, note that for an embedding $\bigsqcup\limits_{i=1}^k B^{2n}(R) \hooksymp \CP^n(1)$, the two ball obstruction gives $a \leq \tfrac{1}{2}$, so Conjecture~\ref{conj:high_d_ball_packing} amounts in this case to:
\begin{align*}
v(\CP^n,k) = 
\begin{cases}
  \tfrac{k}{2^n} & k = 1,\dots,2^{n-1}\\
  1 & k \geq 2^n.
\end{cases}
\end{align*}
This holds for $k = 1,\dots,2^n$, since by \cite{mcduff1994symplectic} we have $v(\CP^n,k^n) = 1$ for all $k \in \Z_{\geq 1}$, and in particular we have an embedding $\bigsqcup\limits_{i=1}^k B^{2n}(1/2-\eps) \hooksymp \CP^n(1)$ for $k = 1,\dots,2^n$ and any $\eps > 0$.
Furthermore, it was also shown in \cite{buse2011symplectic} that $v(\CP^n,k) = 1$ for all $k \geq M_n$, where we can take $M_n = \lceil (8\tfrac{1}{36})^{\tfrac{n}{2}}\rceil$ (c.f. \cite[Rmk. 7.1.31]{mcduff2017introduction}).

\section{Stabilized Gromov two ball theorem}
\begin{proof}[Proof of Theorem \ref{thm:stabilized_Gromov_two_ball}]
We consider the two embeddings given by $f_i: B^{2n}(R_i) \times N \hooksymp \CP^n(R) \times N$ for $i=1,2$. Let $\alpha$ denote the homology class $[\CP^n] \times [\pt]$, and let $\beta$ denote the point class in $\CP^n\times N$. Then a straightforward generalization of theorem \ref{thm:stab_hol_curve} shows
\[
GW_{0,2,A\times \pt}^{\CP^n(R)\times N}(\alpha,\beta) =1.
\]
This essentially comes from the fact in $\CP^n$ for any two points there is a unique curve passing through them. 
Now choose $J$ on $\CP^n \times N$ so that its restriction to the image of $f_i$ restricts to a split almost complex structure. Moreover it is the standard almost complex structure on $B^{2n}(R_i)$. Now such a $J$ might not be generic for the Gromov-Witten invariant $GW_{0,2,A\times \pt}^{\CP^n(R)\times N}(\alpha,\beta)$. So take a sequence of generic $J_n$ so that $J_n\rightarrow J$ in the $C^\infty$ norm as $n\rightarrow \infty$. Consider the surfaces $\Sigma_i : = f_i(\{0\} \times N)$. We take a sequence of generic perturbations, so that $\Sigma_i^n \rightarrow \Sigma$ as $n\rightarrow \infty$ in the $C^\infty$ norm.

Then the condition that $GW_{0,2,A\times \pt}^{\CP^n(R)\times N}(\alpha,\beta) =1$ implies there exists a nontrivial $J_n$-holomorphic curve that passes through $\Sigma_1^n$ and $\Sigma_2^n$ for each $n$. Now take $n\rightarrow \infty$, since the homology class $A\times \pt$ is primitive and has the lowest possible energy for $J$-holomorphic curves, no degeneration can happen. Therefore there exists a $J$-holomorphic curve $C$, that passes through both $\Sigma_1$ and $\Sigma_2$. Applying the monotonicity lemma to the $B^4(R_i)$ component in the image of $f_i$ recovers $R_1 +R_2 \leq R$. 
\end{proof}
\bibliography{biblio}{}
\bibliographystyle{amsalpha}

\end{document}